\documentclass[11pt]{amsart}       

\usepackage{amsmath,amssymb,amsthm}
\usepackage{graphicx} 
\usepackage{tikz}
\usepackage{dsfont}
\usepackage{xspace}	
\usepackage{stmaryrd}
\usepackage{mathrsfs}

\newtheorem{df}{Definition}[section]
\newtheorem{prop}{Proposition}[section]

\newtheorem{theo}{Theorem}[section]
\newtheorem{lem}{Lemma}[section]

\newtheorem{rqqq}{Remark}[section]

\newcommand{\N}{\mathbb{N}}
\newcommand{\NN}{\mathbb{N}^*}
\newcommand{\p}{\mathbb{P}}
\newcommand{\E}{\mathbb{E}}
\newcommand{\Q}{\mathbb{Q}}

\newcommand{\cU}{\mathscr{U}}
\newcommand{\cF}{\mathscr{F}}

\newcommand{\bt}{\mathbf{t}}
\newcommand{\T}{\mathbb{T}}
\newcommand{\bp}{\mathbf{p}}
\newcommand{\bq}{\mathbf{q}}
\newcommand{\bn}{\mathbf{n}}

\newcommand{\indic}{\ensuremath{\mathds{1}}\xspace}

\numberwithin{equation}{section}

\begin{document}

\title[Penalization of Marked Galton-Watson trees]{Penalization of Galton Watson trees with marked vertices}
\date{\today}
\author{Romain Abraham}
\address{Romain Abraham, Institut Denis Poisson,
Universit\'{e} d'Orl\'{e}ans,
Universit\'e de Tours, CNRS, France}
\email{romain.abraham@univ-orleans.fr}

\author{Sonia Boulal}
\address{Sonia Boulal, Institut Denis Poisson,
Universit\'{e} d'Orl\'{e}ans,
Universit\'e de Tours, CNRS, France}
\email{sonia.boulal@univ-orleans.fr}

\author{Pierre Debs}
\address{Pierre Debs, Institut Denis Poisson,
Universit\'{e} d'Orl\'{e}ans,
Universit\'e de Tours, CNRS, France}
\email{pierre.debs@univ-orleans.fr}



\thanks{This work is supported by the ANR project "Rawabranch" number ANR-23-CE40-0008.}

\begin{abstract}
We consider a Galton-Watson tree where each node is mar\-ked independently of each others with a probability depending on its out-degree. Using a penalization method, we exhibit new martingales where the number of marks up to level $n-1$ appears. Then, we use these martingales to define new probability measures via a Girsanov transformation and describe the distribution of the random trees under these new probabilities. 
\end{abstract}

\maketitle
\section{introduction}

We consider in this work  a Galton-Watson process $(Z_n)_{n\in\N}$ (and more precisely its genealogical tree) with offspring distribution $\bp:=(\bp(n))_{n\in\N}$. We suppose throughout this work that $\bp$ admits at least a first moment and that it is non-degenerate:
\begin{equation}\label{condp}
\bp(0)+\bp(1)<1~ \text{and}~\mu(\bp)<+\infty,
\end{equation}
where $\mu(\bp):=\sum_{n\ge0}n\bp(n)$ is the mean of $\bp$. If $\mu(\bp)<1$ (resp. $\mu(\bp)=1$, $\mu(\bp)>1$), we say that the offspring distribution is sub-critical (resp. critical, super-critical). In the sub-critical and critical cases, we have almost surely population extinction.

Over the years, conditioning a critical Galton-Watson tree to be large has been considered: large total progeny (see Kennedy \cite{kennedy_galton-watson_1975} and Geiger and Kaufman \cite{geiger_shape_2004}), large number of leaves (see Curien and Kortchemski \cite{curien_random_2014}) and more generally, large number of nodes whose out-degree belongs to some specific subset of integers (see Abraham and Delmas \cite{abraham_local_2013}). In all these works, we can compute the quantity of interest (that we condition to be large) by marking the nodes that satisfy the studied property and then count the number of nodes. In \cite{abraham_local_2017}, Abraham, Bouaziz and Delmas generalized this approach by marking the nodes randomly where, conditionnally on the tree, the nodes are marked independently of each others with a probability that may depend on their out-degree. More precisely, independently of each other, each individual gives birth to $k$ children and is
\begin{itemize}
\item marked with probability $\bp_0(k,1):=\bp(k)\bq(k)$,\\
\item unmarked with probability $\bp_0(k,0):=\bp(k)(1-\bq(k))$,
\end{itemize}
where $\bq:=(\bq(k))_{k\ge0}$ is a sequence of numbers in $[0,1]$ and is called the mark function. The probability distribution $\bp_0$ on $\N\times\{0,1\}$ is called here the marking-reproduction law of the associated marked Galton-Watson (MGW) tree. We always suppose that the condition 
\begin{equation}\label{condq}
\exists k\in\N,\ \bp(k)\bq(k)>0
\end{equation}
holds so that every individual of a  MGW tree is marked with positive probability. They then condition the tree on having a large number of marked nodes and prove a local convergence of this conditioned marked tree toward Kesten's tree (as for the other conditionnings) as the number of marked nodes tends to infinity.

Another way of getting a tree with an "abnormal" number of marked vertices is to make a change of measure via a Girsanov transformation with a martingale which gives more weight to trees with a large number of marks. In order to find such a martingale, a general method called penalization has been introduced by Roynette, Vallois and Yor \cite{roynette_limiting_2006,roynette_penalisations_2006,roynette_brownian_2009} in the case of the one-dimensional Brownian motion to generate new martingales and to define, by change of measure, Brownian-like processes conditioned on some specific zero-probability events. This idea has been applied to Galton-Watson trees in \cite{abraham_penalization_2020} and this paper is a natural extension of this work.

Let us denote by $M_n$ the number of marks until generation $n-1$, for $n\in\mathbb N^*$ and $M_0=0$, and let $(\cF_n)_{n\ge0}$ be the natural filtration generated by the tree up to generation $n$ and marks up to generation $n-1$. Notice that we need to know the number of offspring of a node to decide whether it is marked or not. That is why, if we look at the tree up to generation $n$, we only have information for the marks up to generation $n-1$, which justifies the definition of the filtration $(\cF_n)_{n\ge 0}$.
The aim of this article is the study, for some measurable positive functions $\phi(x,y)$, of the limit
\begin{equation}\label{limpengen}
	\lim_{p\to +\infty}\dfrac{\E\left[\indic_{\Lambda_n}\phi\left(Z_{n+p},M_{n+p}\right)\right]}{\E\left[\phi\left(Z_{n+p},M_{n+p}\right)\right]}
\end{equation}
where $\Lambda_n\in\mathscr F_n$. If this limit exists,  
it has the form $\E[\indic_{\Lambda_n}B_{n}]$ where $\left(B_n\right)_{n\in\N}$ is a positive $\cF_n$-martingale such that $B_0=1$ (see Section 1.2 of \cite{roynette_penalising_2009}).
This enables to define a probability measure $\Q$ by: 
\begin{equation*}
	\forall \Lambda_n\in\cF_n,~\Q\left(\Lambda_n\right)=\E_{\bp_0}[\indic_{\Lambda_n}B_{n}], 
\end{equation*}
where $\E_{\bp_0}$ denotes the expectation with respect to the probability distribution $\p_{\bp_0}$ of a marked Galton-Watson tree with marking reproduction law $\bp_0$  
and we study the distribution of the tree under $\Q$. In what follows, we omit the reference to $\bp_0$ and write simply $\p$ and $\E$.

We begin with the function $\phi\left(Z_p,M_p\right)=M_p^\ell$, for $\ell\geq1$, which indeed gives weight to trees with a large number of marked nodes. Remark that Assumption \eqref{condq} implies that $\E[M_{n+p}^\ell]>0$ for $n+p>0$. We then obtain the following limits:

\begin{theo}\label{thm:limit_polynome}
Let $\bp$ be an offspring distribution satisfying \eqref{condp} and that admits a moment of order $\ell\in\NN$, and let $\bq$ be a mark function satisfying \eqref{condq}. Then, for every $n\in\N$ and every $\Lambda_n\in\cF_n$, we have
\[
\lim_{p\to+\infty}\frac{\E[\indic_{\Lambda_n}M_{n+p}^\ell]}{\E[M_{n+p}^\ell]}=
\begin{cases}
\E[\indic_{\Lambda_n} f_\ell(M_n,Z_n)] & \text{if }\mu<1,\\
\E[\indic_{\Lambda_n}Z_n] & \text{if }\mu=1,\\
\E\left[\indic_{\Lambda_n} \frac{P_\ell(Z_n)}{\mu^{\ell n}}\right] & \text{if }\mu>1,
\end{cases}
\]
where $P_\ell$ is an explicit polynomial function of degree $\ell$ and $f_\ell$ is defined for all $s,z\in\N$ by:
\begin{equation}\label{expf_j}
	f_\ell(s,z):=\dfrac{1}{\xi_\ell}\sum_{i=0}^{\ell}\binom{\ell}{i}s^{\ell-i}\sum_{t_1+\ldots+t_z=i}\binom{i}{t_1\dots t_z}\prod_{j=1}^z\xi_{t_j},
\end{equation}
where $\binom{j}{t_1\cdots t_{z}}$ is the multinomial coefficient and $\xi_j:=\lim_{p\rightarrow+\infty}\E[M_p^j]$ for all $j\ge0$.
\end{theo}

The two martingales obtained for $\mu \ge 1$ are already known and the study under $\Q$ is, as a result, unnecessary (for $\mu=1$, we obtain Kesten's tree and see \cite{abraham_penalization_2020} for $\mu>1$).

Here, the more interesting case is the sub-critical one where the martingale depends of $M_n$. For instance, let us point out that, for $\ell=1$, we get the martingale
\[
f_1(M_n,Z_n)=Z_n+\frac{1}{\xi_1}M_n.
\]
To describe the genealogical tree of $(Z,M)$ under the new probability, we need to introduce the {\it mass} of a node $u$ denoted by $m_u$: this quantity is zero for the root and is transmitted throughout the tree according to certain rules (see Subsection \ref{masses}) and permits us to exhibit the new reproduction law and marking procedure under $\Q$.

More precisely, let us define a random multi-type tree as follows. The types of the nodes run from $0$ to $\ell$ and, at each generation, the sum of the types must be equal to $\ell$. As a result the type of the root is $\ell$. If there is $z$ individuals at generation $n$, they have the respective types $(t_1,\dots,t_z)$ such that $t_1+\dots+t_z=\ell$ with a probability depending on $z$, $M_n$ and the mass of each individual of generation $n$ (see \eqref{tauxchoixuplettype} for the explicit formula).
Moreover, if $u$ is an individual of type $i\in\llbracket1,\ell\rrbracket$ with mass $m_u$, he has $k\in\N$ children and has mark $\eta\in\{0,1\}$ with probability:
\[
\bp_i(k,\eta):=\frac{\bp_0(k,\eta)f_i(m_u+\eta,k)}{f_i(m_u,1)}
\]
	where $\bp_0$ is the reproduction-marking law of the initial tree but also the reproduction-mark law of a 0-type node.
	 We call a weighted tree of type $\ell$ a tree with these probabilities of reproduction and marking.
	 
\begin{theo}
Let $\bp$ be a {\emph sub-critical} offspring distribution satisfying \eqref{condp} and that admits a moment of order $\ell\in\NN$, and let $\bq$ be a mark function satisfying \eqref{condq}. Let $\Q_\ell$ be the probability measure defined by
\[
\forall n\in\N,\ \frac{d\Q_\ell}{d\p}_{|\cF_n}=f_\ell(M_n,Z_n).
\]
Then $\Q_\ell$ is the distribution of a weighted tree of type $\ell$.
\end{theo}

In a second step, we study the function $\phi\left(Z_p,M_p\right)=H_\ell(Z_p)s^{M_p}t^{Z_p}$, with $s,t\in(0,1)$ where $H_\ell$ is the $\ell$-th Hilbert polynomial defined by:
\begin{equation}\label{Hilbertpol}
	H_0(x)=1 \text{ and } H_\ell(x)=\dfrac{1}{\ell!}\prod_{j=0}^{\ell-1}\left(x-j\right),~\text{for }\ell\geq1.
\end{equation}
Contrary to the previous study, this function gives weight to trees with a small number of marks. We obtain the following limits:

\begin{theo}\label{thm:expo}
Let $\bp$ be an offspring distribution satisfying \eqref{condp} and that admits a moment of order $\ell\in\NN$, and let $\bq$ be a mark function satisfying \eqref{condq}. We define
\[
r=\min\{k\in\N,\bp(k)>0\}.
\] 
Then, for every $s\in(0,1)$, if $\kappa(s)$ denotes the only root in $[0,1]$ of the function
\[
f_s(t):=\E[s^{M_1}t^{Z_1}],
\]
we have for every $n\in\N$ and $\Lambda_n\in\cF_n$,
\begin{itemize}
\item if $r=0$, for every $t\in[0,1]$,
\[
\lim_{p\to+\infty}\frac{\E[\indic_{\Lambda_n}H_\ell(Z_{n+p})s^{M_{n+p}}t^{Z_{n+p}}]}{\E[H_\ell(Z_{n+p})s^{M_{n+p}}t^{Z_{n+p}}]}=
\begin{cases}
\E\left[\indic_{\Lambda_n}s^{M_n}\kappa(s)^{Z_n-1}\right] & \text{ if }\ell=0,\\ 
E\left[\indic_{\Lambda_n}\dfrac{s^{M_n}Z_n \kappa(s)^{Z_n-1}}{f_s'(\kappa(s))^n}\right]& \text{ if }\ell>0,
\end{cases}
\]
\item if $r=1$, for every $t\in (0,1]$,
\[
\lim_{p\to+\infty}\frac{\E[\indic_{\Lambda_n}H_\ell(Z_{n+p})s^{M_{n+p}}t^{Z_{n+p}}]}{\E[H_\ell(Z_{n+p})s^{M_{n+p}}t^{Z_{n+p}}]}=\E\left[\indic_{\Lambda_n}s^{M_n}\alpha_1(s)^{-n}\indic_{Z_n=1}\right],
\]
\item if $r\ge 2$, for every $t\in (0,1]$,
\[
\lim_{p\to+\infty}\frac{\E[\indic_{\Lambda_n}H_\ell(Z_{n+p})s^{M_{n+p}}t^{Z_{n+p}}]}{\E[H_\ell(Z_{n+p})s^{M_{n+p}}t^{Z_{n+p}}]}=\E\left[\indic_{\Lambda_n}s^{M_n}\alpha_r(s)^{\frac{-\left(r^n-1\right)}{r-1}}\indic_{Z_n=r^{n}}\right],
\]
\end{itemize}
	where $\alpha_r(s)=\bp(r)(s\bq(r)+1-\bq(r))$.
\end{theo}

Let us add that, for sake of completeness, we also study the case $s=0$ in Theorem \ref{pen0MntZnp0}.

We also describe the probability distribution obtained by a change of measure using the martingales obtained in the previous theorem, see Theorems \ref{pensMntZn1}, \ref{probunderQ_(s,k)rn}, \ref{thm:cond_expo_r} and \ref{pen0MntZnp0}.

\medskip
The paper is organized as follows: in Section \ref{techback} we introduce the set of discrete marked trees and define the MGW trees. We study in Section \ref{sec:polynomial} the penalization by $M_p^\ell$, finding the martingale obtained at the limit and then describing the distribution obtained by a Girsanov transformation using this martingale. We make the same study in Section \ref{sec:expo} when the penalization is of the form $H_\ell(Z_p)s^{M_p}t^{Z_p}$.

\section{Technical background}\label{techback}

\subsection{The set of marked trees}\label{settree}
Let $\N$ be the set of nonnegative integers and $\NN$ the set of positive integers.
We recall Neveu's formalism \cite{neveu_arbres_1986}  for ordered rooted trees. We set 
$\cU=\bigcup_{n\geq 0}(\NN)^n$
the set of finite sequences of positive integers with the convention $(\NN)^0=\{\emptyset\}$. For $n\geq 0$ and $u=(u_1,...,u_n)\in\cU$, let $|u|=n$ be the length of $u$ with the convention $|\emptyset|=0$. If $u$ and $v$ are two sequences of $\cU$, we denote by $uv$ their concatenation,
 with the convention that $uv=u$ if $v=\emptyset$ and $uv=v$ if $u=\emptyset$.

\medskip
A tree $\bt$ is a subset of $\cU$ that satisfies:
\begin{itemize}
	\item $\emptyset \in t$;
	\item $\forall u\in\cU$, $\forall j\in\NN$, $uj\in\bt\Longrightarrow u\in\bt$;
	\item $\forall u\in \bt$, $\exists k_u(\bt)\in\N$, $\forall i\in\NN$, $ui\in\bt\iff i\le k_u(\bt)$.
\end{itemize}
The integer $k_u(\bt)$ represents the number of offspring of the vertex $u\in \bt$.  
We denote by $\T$ the set of trees.

\medskip
A marked tree $\bt^*$ is defined by a tree $\bt\in \T$ and a mark $\eta_u\in\{0,1\}$ for every node $u\in\bt$, that is
\[
\bt^*=\bigl(\bt,(\eta_u)_{u\in\bt}\bigr).
\]
A node $u\in\bt$ is said to be marked if $\eta_u=1$ and unmarked if $\eta_u=0$.
We denote by $\T^*$ the set of marked trees.

\medskip
For every $h\in \N$, we define the restriction functions
\[
r_h:\T\longrightarrow\T,\qquad\text{and}\qquad r_h^*:\T^*\longrightarrow\T^*
\]
by, for $\bt\in\T$, $\bt^*=(\bt, (\eta_u,u\in\bt))\in\T^*$,
\begin{align*}
r_h(\bt)=\{u\in\bt,\ |u|\le h\},\qquad
r_h^*(\bt^*)=\bigl(r_h(\bt), (\eta_u^h)_{u\in r_h(\bt)}\bigr)
\end{align*}
with 
\[
\eta_u^h=\begin{cases}
\eta_u & \text{if } |u|\le h-1,\\
0 & \text{if }|u|=h.
\end{cases}
\]	

\begin{rqqq}
Let us stress that the nodes at level $h$ of $r_h^*(\bt^*)$ are unmarked so that the tree $r_h^*(\bt^*)$ is completely characterized by $((k_u(\bt),\eta_u),u\in\bt,|u|\le h-1)$.
\end{rqqq}

We can endow the set $\T^*$ with the $\sigma$-field $\cF$ generated by the family of sets $(\{\bt^*\in\T^*,\ u\in\bt\})_{u\in\cU}$ and hence define probability measures on $(\T^*,\cF)$.

We also endow $\T^*$ with the filtration $(\cF_n)_{n\ge 0}$ where $\cF_n$ is the $\sigma$-field generated by the restiction function $r_n^*$. Notice that $\cF=\bigvee_{n\ge 0} \cF_n$ as, for every $u\in\cU$,
\[
\{\bt^*\in\T^*,\ u\in\bt\}\in\cF_{|u|}.
\]

\medskip
Let $\bt^*=(\bt,(\eta_u)_{u\in\bt})\in\T^*$ be a marked tree. For every $n\in\N$, we define
\begin{itemize}
\item $z_n(\bt^*)$ the nodes of $\bt$ at generation $n$:
\[
z_n(\bt^*)=\{u\in\bt,\ |u|=n\},
\]
\item $Z_n(\bt^*)$ the size of generation $n$:
\[
Z_n(\bt^*)=\mathrm{Card}\bigl(z_n(\bt^*)\bigr),
\]
\item $M_n(\bt^*)$ the total number of marked nodes up to generation $n-1$:
\[
M_0(\bt^*)=0 \quad \text{ and }\quad M_n(\bt^*)=\sum_{u\in\bt,|u|\le n-1}\eta_u \text{ for }n\ge 1.
\]
\end{itemize}
Let us point out that the functions $z_n,\ Z_n,\ M_n$ are all $\cF_{n}$-measurable.

\subsection{Masses}\label{masses}
Let $\bt^*=(\bt,(\eta_u)_{u\in\bt})\in\T^*$ be a marked tree. We then define for every node $u\in\bt$ a new quantity called the mass of the node $u$ and denoted by $m_u$. We set $m_\emptyset=0$ and the masses are then defined recursively as follows. Let $n\in\N$ and let us suppose that the masses $(m_u,|u|\le n)$ are constructed. Let $v\in z_{n+1}(\bt^*)$ that we write $v=uj$ with $u\in\bt$ and $1\le j\le k_u(\bt)$. Then, we set
\begin{equation}\label{valmass}
	m_{uj}=\frac{m_u+\eta_u}{k_u(\bt)}+\frac{ \sum_{w\in z_n(\bt^*),k_w(\bt)=0}\left(m_w+\eta_w\right) }{Z_{n+1}(\bt)}\cdot
\end{equation}
This formula can be interpreted as follows:
\begin{itemize}
	\item The first term means that $u$ transmits its mass, plus one if it is marked, uniformly to all its children;
	\item The second term means that each childless node at generation $n$ transmits its mass, plus one if it is marked, uniformly to all nodes of generation $n+1$.
	\end{itemize}
	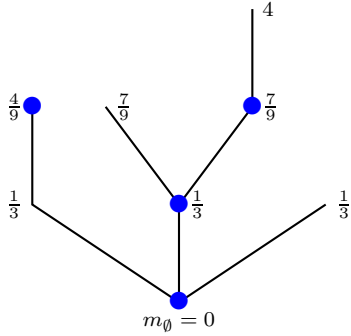
\begin{figure}[ht!]
		\centering
		\begin{tikzpicture}[scale=1.3]
			\coordinate (0) at (0,0);
			\coordinate (1) at (-1.5,1);
			\coordinate (2) at (0,1);
			\coordinate (3) at (1.5,1);
			\coordinate (11) at (-1.5,2);
			\coordinate (21) at (-0.75,2);
			\coordinate (22) at (0.75,2);
			\coordinate (221) at (0.75,3);
			\coordinate (Z1) at (3,1);
			\coordinate (Z2) at (3,2);
			\coordinate (Z3) at (3,3);
			\draw [thick](0) -- (2) -- (21);
			\draw [thick](0)-- (3) ;
			\draw [thick](0) -- (1)-- (11);
			\draw [thick](2) -- (22)-- (221);
			\draw (0)node{\textcolor{blue}{\LARGE{$\bullet$}}}node[below]{\scriptsize{$m_\emptyset=0$}} ;
			\draw (2)node{\textcolor{blue}{\LARGE{$\bullet$}}};
			\draw (11)node{\textcolor{blue}{\LARGE{$\bullet$}}};
			\draw (22)node{\textcolor{blue}{\LARGE{$\bullet$}}};
			\draw (1) node[left]{\scriptsize{$\frac{1}{3}$}};
			\draw (2) node[right]{\scriptsize{$\frac{1}{3}$}}  ;
			\draw (3) node[ right]{\scriptsize{$\frac{1}{3}$}} ;
			\draw (11) node[left]{\scriptsize{$\frac{4}{9}$}}  ;
			\draw (21) node[right]{\scriptsize{$\frac{7}{9}$}};
			\draw (22) node[right]{\scriptsize{$\frac{7}{9}$}}  ;
			\draw (221) node[right]{\scriptsize{$4$}};
		\end{tikzpicture}
		\caption{A marked tree and the associated masses.}
		\label{exlimtreemartim}
	\end{figure}
Notice that the map $\bt^*\mapsto (m_u,u\in z_{n}(\bt^*))$ is $\cF_n$-measurable.

Moreover, by construction, we have the following link between the masses and the number of marked nodes up to some level:
	\begin{equation}\label{expM_n}
\forall n\in\N,\qquad \sum_{u\in z_n(\bt^*)}m_u = M_n(\bt^*).
	\end{equation}

\subsection{Marked Galton-Watson trees}

Let $\bp_0=(\bp_0(k,\eta),k\in\N,\eta\in\{0,1\})$ be a probability distribution on $\N\times\{0,1\}$. There exists a unique probability measure $\p_{\bp _0}$ on $(\T^*,\cF)$ such that, for all $h\in\NN$ and all $\bt^*=(\bt,(\eta_u)_{u\in\bt})\in\T^*$, if $\tau^*$ denotes the canonical random variable on $\T^*$ (i.e. the identity map of $\T^*$),
\begin{equation}\label{def_p}
\p_{\bp_0}(r_h^*(\tau^*)=r_h^*(\bt^*))=\prod_{u\in r_{h-1}(\bt)} { \bp_0(k_u(\bt),\eta_u)}.
\end{equation}
We say that the r.w. $\tau^*$ under $\p_{\bp_0}$ is a marked Galton-Watson (MGW) tree with reproduction-marking law $\bp_0$.

Equivalently, the distribution of a marked Galton-Watson tree with re\-pro\-duc\-tion-marking law $\bp_0$ is also characterized by an offspring distribution $\bp$ and a mark function $\bq:\N\longrightarrow [0,1]$ with
\[
\bp(k)=\bp_0(k,0)+\bp_0(k,1),\quad \bq(k)=
\frac{\bp_0(k,1)}{\bp(k)} \text{ if }\bp(k)\ne 0
\]
(the value of $\bq(k)$ has no importance if $\bp(k)=0$), or equivalently
\begin{equation}\label{def_p0}
\bp_0(k,1)=\bp(k)\bq(k),\qquad \bp_0(k,0)=\bp(k)(1-\bq(k)).
\end{equation}
This approach (giving $\bp$ and $\bq$) consists in first picking a Galton-Watson tree with offspring distribution $\bp$ then conditionnaly given the tree, adding marks on every node, independently of each others, with probabily $\bq(k)$ where $k$ is the out-degree of the node.

We will then write $\p_{\bp_0}$ or $\p_{\bp,\bq}$ depending on the adopted point of view, or simply $\p$ when the context is clear.

\medskip
Under $\p_{\bp,\bq}$, the random tree $\tau^*$ satisfies the so-called branching property that is
\begin{itemize}
\item The random variable $k_\emptyset(\tau^*)$ is distributed according to the probability distribution $\bp$,
\item Conditionally given $\{k_\emptyset(\tau^*)=j\}$, the root is marked with probability $\bq(j)$,
\item Conditionally on $\{k_\emptyset(\tau^*)=j\}$, the $j$ sub-trees attached to the root are i.i.d. random marked trees distributed according to the probability $\p_{\bp,\bq}$.
\end{itemize}

As a consequence, if we set $Z_n:=Z_n(\tau^*)$ and $M_n:=M_n(\tau^*)$, we have for every $n,p\in\N^*$, the following equality in distribution under $\p_{\bp,\bq}$,
\begin{equation}\label{expM_p}
	M_{n+p}\overset{(d)}{=}M_n+\sum_{i=1}^{Z_n}{\tilde{M}_{i,p}},
\end{equation}
where $(\tilde{M}_{i,p})_{i\geq 1}$ is a sequence of i.i.d copies of  $M_{p}$ and independent of $\cF_n$.

\section{Polynomial penalization}\label{sec:polynomial}

In this section, we consider an offspring distribution $\bp$ satisfying \eqref{condp} and that admits a moment of order $\ell\in\NN$, and a mark function $\bq$ satisfying \eqref{condq}. We denote by $\mu$ the mean of $\bp$. We study the penalization by $M_{n+p}^\ell$ i.e. we look for the limit for all $n\in\N$ and $\Lambda_n\in\cF_n$
\[
\lim_{p\to+\infty}\frac{\E[\indic_{\Lambda_n}M_{n+p}^\ell]}{\E[M_{n+p}^\ell]}=\lim_{p\to+\infty}\frac{\E[\indic_{\Lambda_n}\E[M_{n+p}^\ell|\cF_n]]}{\E[M_{n+p}^\ell]}\cdot
\]

To begin with, let us prove that this expression makes sense. Remark first that \eqref{condq} implies that $\E[M_{n+p}^\ell]\ne 0$ for $p\ge 1$.

\begin{lem}\label{propmomordk}
Let $\bp$ be an offspring distribution that admits an moment of order $\ell\in\NN$ and let $\bq$ be a mark function. Then, under $\p$, for every $p\in\N$, $Z_p$ and $M_p$ have a moment of order $\ell$.
\end{lem}
\begin{proof} We begin to prove that $Z_p$ has a moment of order $\ell$ by induction on $p$. The result is clear for $p=0$, as $Z_0=1$.
\smallbreak
 Let $p\in \N$ be such that $Z_j$ has a moment of order $\ell$ for all $j\leq p$. Let $X$, $(X_i)_{i\ge 1}$ be i.i.d. random variables with distribution $\bp$, independent of $\cF_p$.
We have:
\begin{align*}
	\E\left[Z_{p+1}^\ell\right]&=\E\left[\E\left[\left. \left(\sum_{i=1}^{Z_p}X_{i}\right)^\ell \right| \cF_p\right]\right]\\
	& =\E\left[\sum_{t_1+\ldots+t_{Z_p}=\ell}\binom{\ell}{t_1\cdots t_{Z_p}}\prod_{\ell=1}^{Z_p}\E\left[X^{t_{\ell}}\right]\right]\\
&\leq \E\left[\sum_{t_1+\ldots+t_{Z_p}=\ell}\binom{\ell}{t_1\cdots t_{Z_p}}
\E\left[X^{\ell}\right]^\ell\right]=\E\left[Z_p^\ell\right]\E\left[X^{\ell}\right]^\ell
\end{align*}
where in the last inequality, we use that there are at most $\ell$ non-zero $t_i$'s.
As $Z_p^\ell$ and $X^\ell$ are in $L^1$ by assumption, we have our first result. 

\medskip
The previous result implies that $\sum_{i=0}^{p-1}Z_i\in \mathrm{L}^\ell$ and, as $M_p\leq \sum_{i=0}^{p-1}Z_i$, we get the second result.
\end{proof}

We now give an expression of $\E\left[\left.M_{n+p}^\ell\right|\cF_n\right]$ that will be used troughout this section. Recall the equality in distribution \eqref{expM_p}. Using the same notations, we get for all $p,n\in \N$:
\begin{align}\label{expM_p^k}
\E\left[M_{n+p}^\ell\bigm| \cF_n\right] & =\sum_{j=0}^\ell\binom{\ell}{j}\E\left[M_n^{\ell-j}\left(\sum_{i=1}^{Z_n}\tilde{M}_{i,p}\right)^j\biggm| \cF_n\right]\nonumber\\
& =\sum_{j=0}^\ell\binom{\ell}{j}M_n^{\ell-j}\sum_{t_1+\cdots+t_{Z_n}=j}\binom{j}{t_1\cdots t_{Z_n}}\prod_{i=1}^{Z_n}\E[M_p^{t_i}].
\end{align}

\subsection{The sub-critical case $\mu<1$}

\subsubsection{A new martingale}

Before stating the main result of this section, let us state some integrability lemmas and introduce some notations.

\begin{lem}\label{momentkzinfty} Let $\bp$ be a sub-critical offspring distribution that admits a moment of order $\ell\in \NN$.  Let $\tilde Z_\infty$ be the total population size of the associated Galton-Watson tree i.e.
\[
\tilde Z_\infty=\mathrm{Card}(\tau)=\sum_{n=0}^{+\infty} Z_n.
\]
Then $\tilde{Z}_\infty$ has a moment of order $\ell$.
\end{lem}

\begin{proof}
Let us denote by $g$ (resp. $G_\infty$) the generating function of $\bp$ (resp. $\tilde Z_\infty$). Then we have, with $(\tilde Z_{i,\infty})_{i\ge 1}$ i.i.d. copies of $\tilde Z_\infty$, for every $s\in [0,1]$,
	\[
			G_\infty(s)
			=\E\left[s^{1+\sum_{i=1}^{Z_1}\tilde{Z}_{i,\infty}}\right]
			=sg(G_\infty(s)).
	\]

Differentiating this relation, we obtain, for every $s\in[0,1)$,
		\begin{equation}\label{expGinftyder}
		G_\infty'(s)=\frac{g(G_\infty(s))}{1-sg'(G_\infty(s))}.
	\end{equation}
	Then, letting $s\to 1$, we have
	\begin{equation*}
		\lim_{s\to 1}G_\infty'(s)=\frac{1}{1-\mu}<+\infty.
	\end{equation*}
	Therefore, $\tilde{Z}_\infty$ admits a moment of order $1$, $\E[\tilde{Z}_\infty] =\frac{1}{1-\mu}$ and $G_\infty$ is of class $\mathscr{C}$ on $[0,1]$.
	
An obvious induction argument gives that, for every $i\le \ell$, there exists a polynomial function $P_i$ such that for every $s\in[0,1)$,
\[
G^{(i)}_\infty(s)=\frac{P_i(s,g(G_\infty(s)),\ldots,g^{(i)}(G_\infty(s)),G_\infty(s),G'_\infty(s),\ldots G_\infty^{(i-1)}(s))}{(1-sg'(G_\infty(s))^i}
\]
and again by induction we get that $G_\infty^{(i)}(s)$ admits a limit as $s\to 1$ which implies that $\tilde Z_\infty$ admits moments up to order $\ell$	
\end{proof}

Let us denote by $M_\infty$ the total number of marks on the tree. As $M_\infty\le \tilde Z_\infty$ a.s., $M_\infty$ also admits a moment of order $\ell$. We set for every $i\le \ell$
\begin{equation}
\xi_i=\E[M_\infty^i].
\end{equation}

The constants $\xi_i$ can be explicitely computed using the following recursion equation:
\begin{prop}\label{expxi_k}
	Let $\bp$ be an offspring distribution such that $\mu<1$, with a moment of order $\ell\in\NN$ and let $\bq$ be a mark function. Then we have
	\begin{multline*}
			(1-\mu)\xi_\ell=\sum_{j=0}^{\ell-1}\binom{\ell}{j}\E\left[M_1\sum_{t_1+\cdots +t_{Z_1}=j}\binom{j}{t_1\cdots t_{Z_1}}\prod_{i=1}^{Z_1}\xi_{t_i}\right]\\
			+\E\left[\sum_{\underset{t_i<\ell,\forall i}{t_1+\ldots+t_{Z_1}=\ell}}\binom{\ell}{t_1\cdots t_{Z_1}}\prod_{i=1}^{Z_1}\xi_{t_i}\right].
	\end{multline*}
\end{prop}

\begin{proof} 
First notice that, as $M_1$ is equal to 0 or 1, we have $M_1^j=M_1$ for every $j\ge 1$.
Using \eqref{expM_p^k}, we can write:
\begin{align*}
\E[M_{p+1}^\ell]-\mu\E[M_{p}^\ell]&=\sum_{j=0}^{\ell-1}\binom{\ell}{j}\E\left[M_1\sum_{t_1+ \ldots +t_{Z_1}=j}\binom{j}{t_1\dots t_{Z_1}}\prod_{i=1}^{Z_1}\E[M_p^{t_i}]\right]\\
&\qquad+\E\left[\sum_{\underset{t_i<\ell,\forall i}{t_1+\ldots+t_{Z_1}=\ell}}\binom{j}{t_1\dots t_{Z_1}}\prod_{i=1}^{Z_1}\E[M_p^{t_i}]\right]\\
&:=\sum_{j=0}^{\ell-1}\binom{\ell}{j}\E[A_{j,p}]+\E[B_p].
\end{align*}
As $(A_{j,p})_{p\ge 0}$ and $(B_p)_{p\ge 0}$ are non-negative and non-decreasing sequences, we obtain $\xi_\ell$ taking the limit when $p$ goes to infinity and using Beppo Levi's theorem. 
\end{proof}

We can now state the main result of this section. Let us recall, for $\ell\in\N$,  the definition of the function $f_\ell$ defined for all  $s,z\in\N$ by:
\begin{equation}\label{def_fl}
	f_\ell(s,z):=\frac{1}{\xi_\ell}\sum_{i=0}^{\ell}\binom{\ell}{i}s^{\ell-i}\sum_{t_1+\ldots+t_z=i}\binom{i}{t_1\dots t_z}\prod_{j=1}^z\xi_{t_j}
\end{equation}
with the convention $f_\ell(s,0)=\frac{s^\ell}{\xi_\ell}$.
Notice that this definition is also valid for $\ell=0$ and we get $f_0(s,z)=1$.
\begin{theo}\label{limcondmu<1k}
Let $\bp$ be an offspring distribution satisfying \eqref{condp}, that admits a moment of order $\ell\in\NN$ and such that $\mu<1$, and let $\bq$ be a mark function satisfying \eqref{condq}. For every $n\in \N$ and every $\Lambda_n\in\cF_n$, we have
	\begin{equation}\label{limcondmartk}
		\lim_{p\to +\infty}\frac{\E\left[\indic_{\Lambda_n}M_{n+p}^\ell\right]}{\E\left[M_{n+p}^\ell\right]}=\E[\indic_{\Lambda_n}f_\ell(M_n ,Z_n)].
	\end{equation}
	\end{theo}

\begin{proof}
Let $p,n\in\N$. Using \eqref{expM_p^k}, we get:
\begin{multline*}
		\E\left[\mathds{1}_{\Lambda_n}M_{n+p}^\ell\right]\\
=\E\left[\mathds{1}_{\Lambda_n}	\sum_{j=0}^\ell\binom{\ell}{j}M_n^{\ell-j}\sum_{t_1+\ldots+t_{Z_n}=j}\binom{j}{t_1\cdots t_{Z_n}}\prod^{Z_n}_{i=1}\E\left[M_{p}^{t_{i}}\right]\right].
\end{multline*}
Using again Beppo Levi's theorem, we obtain
\[\lim_{p\rightarrow+\infty}\E\left[\mathds{1}_{\Lambda_n}M_{n+p}^\ell\right]=\xi_\ell\E[\mathds{1}_{\Lambda_n}f_\ell(M_n,Z_n)]\]
and we conclude using
\[
\lim_{p\to+\infty}\E[M_{n+p}^\ell]=\xi_\ell.
\]
\end{proof}

\begin{rqqq}
For $\ell=1$, the obtained martingale is
\[
H_{n,1}=Z_n+\frac{1}{\xi_1}M_n.
\]
\end{rqqq}

\subsubsection{Change of probability and the new random tree}

We still consider a subcritical offspring distribution $\bp$ satisfying \eqref{condp} and that admits a moment of order $\ell\in\NN$, and a mark function $\bq$ satisfying \eqref{condq}. We define a new probability measure $\Q_\ell$ on $(\T^*,\cF)$ by
\begin{equation}\label{defQk}
\forall n\in\N,\ \forall \Lambda_n\in\cF_n,\ \Q_\ell(\Lambda_n)=\E[\indic_{\Lambda_n}f_\ell(Z_n,M_n)].
\end{equation}

To describe this probability measure $\Q_\ell$, let us define a $\T^*$-valued random tree $\tau_\ell(\bp,\bq)$ defined on some probability space $(\Omega,\mathscr{A},P)$.

\begin{df}\label{def_tauk}
The tree $\tau_{\ell}(\bp,\bq)$ is a multitype random marked tree that is defined recursively as follows:
\begin{itemize}
\item The types of the nodes run from $0$ to $\ell$;
\item The type of the root is $\ell$;
\item Let us suppose that the tree $r_n^*(\tau_\ell(\bp,\bq))$ is constructed and that the types of the nodes at generation $n$ are known for some $n\ge 0$. Then all individuals at generation $n$ reproduce and are marked independently of each other conditionally given $r_n^*(\tau_\ell(\bp,\bq))$ according to its type and its mass (recall the definition of the mass of a node introduced in Section \ref{masses}):
\begin{itemize}
\item an individual $u$ of type $0$, has $k\in\N$ children and a mark $\eta\in\{0,1\}$ with probability:
\[
p_{0}(k,\eta):=\begin{cases} \bp(k)(1-\bq(k)) & \text{if }\eta=0,\\
\bp(k)\bq(k) & \text{if } \eta=1,
\end{cases}
\]
\item an individual $u$ of type $i\in\llbracket 1,\ell\rrbracket$ and with mass $m_u$, has $k \in\N$ children and a mark $\eta\in\{0,1\}$ with probability:
\[
p_{i}(k,\eta):=\frac{p_{0}(k,\eta)f_i(m_u+\eta,k)}{f_i(m_u,1)};
\]
\end{itemize}	
\item We then choose the types of the nodes at generation $n+1$. At each generation, the sum of the types must be equal to $\ell$. Moreover, conditionally given $z_{n+1}(\tau_\ell(\bp,\bq))=z_{n+1}$, the probability that the nodes at generation $n+1$ have respective types $(t_u,u\in z_{n+1})$ with $\sum_{u\in z_{n+1}}t_u=\ell$, is 
\begin{equation}\label{tauxchoixuplettype}
	\gamma_{\ell,(t_u)_{u\in z_{n+1}}}:=\binom{\ell}{t_u,u\in z_{n+1}}\frac{\prod_{u\in z_{n+1}}f_{t_u}(m_u,1)\xi_{t_u}}{\xi_\ell f_\ell(M_{n+1},Z_{n+1})},
\end{equation}
with the notation
\begin{equation*}
	\binom{\ell}{t_u,u\in z_{n+1}}:= \frac{\ell!}{\prod_{u\in z_{n+1}}t_u!}.
\end{equation*} 
for the usual multinomial coefficient.
\end{itemize}
\end{df}

\begin{rqqq}
The mass $m_u$ of a node $u$ (and hence its offspring distribution) depends on all extinct nodes in the past and not only on the ancestors of the node. Moreover, the choice of the types of the nodes at some generation depends on the size of this generation. Consequently, the tree $\tau_\ell$ does not satisfy the branching property.
\end{rqqq}

\begin{rqqq}
If all the nodes below generation $n$ are not marked, all the masses $m_u$ for $u\in z_n$ are null and hence, for every node $u\in z_n$ with type $i\ge 1$, we have
\[
p_i(0,0)=0.
\]
Therefore, the tree $\tau_\ell(\bp,\bq)$ cannot die out without any marked node.
\end{rqqq}

 \begin{lem}\label{relf_k}
Formula \eqref{tauxchoixuplettype} indeed defines a probability distribution on $\llbracket 0,\ell\rrbracket ^{Z_{n+1}}$.
\end{lem}
\begin{proof}
	This result is equivalent to show that for all $n\in\N$, we have
	\begin{equation}\label{relf_kf_t}
		f_\ell(M_n,Z_n)=\frac{1}{\xi_\ell}\sum_{\sum_{u\in z_n}t_u=\ell}\binom{\ell}{t_u,u\in z_n}\prod_{u\in z_n}f_{t_u}(m_u,1)\xi_{t_u}.
	\end{equation}
	To obtain \eqref{relf_kf_t}, we need to obtain an alternate expression for the a.s. limit of $\E[M_{n+p}^\ell|\cF_n]$, as we already know that this limit is $\xi_\ell f_\ell(M_n,Z_n)$ a.s. For this purpose, according to \eqref{expM_n} and the branching property, conditionally on $\cF_n$, we have: 
\[M_{n+p}\overset{(d)}{=}\sum_{u\in z_n}m_u+M_{u,p},\]
where the processes $(M_{u,.}, u\in \mathscr U)$ are i.i.d. copies of $M$, independent of $\cF_n$. As a result: 
\begin{align*}
\E[M_{n+p}^\ell|\cF_n]&=\sum_{\sum_{u\in z_n}t_u=\ell}\binom{\ell}{t_u,u\in z_n}\prod_{v\in z_n}\E\left[(m_v+M_{v,p})^{t_v}|\cF_n\right]\\
&=\sum_{\sum_{u\in z_n}t_u=\ell}\binom{\ell}{t_u,u\in z_n}\prod_{v\in z_n}\sum_{j=0}^{t_v}\binom{t_v}{j}m_v^{t_v-j}\E\left[M_{p}^j\right].
\end{align*}
Using Beppo Levi's theorem:
\[\lim_{p\rightarrow+\infty}\E[M_{p+n}^\ell|\cF_n]=\sum_{\sum_{u\in z_n}t_u=\ell}\binom{\ell}{t_u,u\in z_n}\prod_{v\in z_n}\sum_{j=0}^{t_v}\binom{t_v}{j}m_v^{t_v-j}\xi_j,\]
and we easily conclude as, thanks to \eqref{def_fl}, 
\[
f_{t_v}(m_v,1)\xi_{t_{v}}=\sum_{j=0}^{t_v}\binom{t_v}{j}m_v^{t_v-j}\xi_j.
\]
\end{proof}

In order to describe the new probability $\Q_\ell$, we need an alternative expression of $f_\ell(M_{{n+1}},Z_{n+1})$: 
\begin{lem}\label{relf_k_n+1}
	For all $n\in\N$, we have
	\begin{equation}\label{relf_k,n+1f_t}
		f_\ell(M_{{n+1}},Z_{n+1})=\frac{1}{\xi_\ell}\sum_{\sum_{u\in z_n}t_u=\ell}\binom{\ell}{t_u,u\in z_n}{\prod_{u\in z_n}}f_{t_u}(m_u+\eta_u,k_u)\xi_{t_u},
	\end{equation}
	with $k_u$ the number of children of the individual $u$ and $\eta_u$ its mark.
\end{lem}
\begin{proof} The proof is similar to the proof of Lemma \ref{relf_k}: we write 
\[M_{p+n}=\sum_{u\in z_n}\left(m_u+\eta_u+\sum_{i=1}^{k_u}M_{i,p-1}\right)\]
where the processes $(M_{i,.}, i\ge0)$ are i.i.d. copies of $M$ independent of $\cF_{n+1}$. Thus, $\E[M_{n+p}^\ell|\cF_{n+1}]$ is equal to:
\begin{multline*}
\sum_{\sum_{u\in z_n}t_u=\ell}\binom{\ell}{t_u,u\in z_n}\prod_{v\in z_n}\sum_{j=0}^{t_v}\binom{t_v}{j}(m_v+\eta_v)^{t_v-j}\E\left[\left(\left.\sum_{i=1}^{k_u}M_{i,p-1}\right)^j\right|\cF_{n+1}\right]\\
=\sum_{\sum_{u\in z_n}t_u=\ell}\binom{\ell}{t_u,u\in z_n}\prod_{v\in z_n}\sum_{j=0}^{t_v}\binom{t_v}{j}(m_v+\eta_v)^{t_v-j}\\
\sum_{r_1+\dots+r_{k_u}=j}\binom{j}{r_1\dots r_{k_u}}\prod_{i=1}^{k_u}\E[M_{i,p-1}^{r_i}].
\end{multline*}
Taking the limit when $p$ goes to infinity, thanks to Beppo Levi's theorem, we obtain \eqref{relf_k,n+1f_t}.
\end{proof}

\begin{theo}\label{probunderQ_krn}
The probability $\Q_\ell$ defined by \eqref{defQk} is the distribution of the tree $\tau_\ell(\bp,\bq)$ of Definition \ref{def_tauk}.
\end{theo}

\begin{proof}
To prove the theorem, it suffices to prove that for all $n\in\N$, and all $\bt^*\in \mathbb T^*$:
	\begin{equation}\label{probunderQ_krneq}
		\Q_\ell(r^{*}_{n}\left(\tau^{*})=r^{*}_{n}(\bt^{*})\right)=P(r^{*}_{n}(\tau_{\ell})=r^{*}_{n}(\bt^{*}))
	\end{equation}
	where $\tau^*$ is the canonical random variable on $\T^*$ and where we set $\tau_\ell$ instead of $\tau_\ell(\bp,\bq)$.
First, note that:
\begin{align}\label{toto}
		\Q_\ell(r^{*}_{n}(\tau^{*})=r^{*}_{n}(\bt^{*}))&=\E\left[\indic_{r^{*}_{n}(\tau^{*})=r^{*}_{n}(\bt^{*})}f_\ell(M_n(\tau^*),Z_n(\tau^*)\right]\nonumber\\
		&=f_\ell(M_n(\bt^*),Z_n(\bt^*))\p(r^{*}_{n}(\tau^{*})=r^{*}_{n}(\bt^{*})).
\end{align}

We prove \eqref{probunderQ_krneq} by induction on $n$.

For $n=0$, \eqref{probunderQ_krneq} is true as, for every tree $\bt^*$ (and hence also for any random tree), $r_0^*(\bt^*)$ is the tree reduced to the root with mark 0.

	Assume that \eqref{probunderQ_krneq} is true for $n\in\N$. Using the recursive definition of $\tau_\ell$, we have: 
\begin{multline*}
			P(r^{*}_{n+1}(\tau_{\ell})=r^{*}_{n+1}(\bt^{*}))\\
			=P(r^{*}_{n}(\tau_{\ell})=r^{*}_{n}(t^{*}))\prod_{u\in z_n(\bt^*)}p_0(k_u(\bt),\eta_u(\bt^*))\\
			\times\sum_{\sum_{u\in z_n(\bt^*)}t_u=\ell}\gamma_{\ell,(t_u)_{u\in z_n(\bt^*)}}\prod_{u\in z_n(\bt^*)}\frac{f_{t_u}(m_u(\bt^*)+\eta_u(\bt^*),k_u(\bt))}{f_{t_u}(m_u(\bt^*),1)}\cdot
\end{multline*}
Then, the induction assumption as well as the definition of the probabilities $\gamma$ give 
\begin{multline*}
P(r^{*}_{n+1}(\tau_{\ell})=r^{*}_{n+1}(\bt^{*}))\\
=\Q_\ell(r^{*}_{n}(\tau^{*})=r^{*}_{n}(\bt^{*}))\prod_{u\in z_n(\bt^*)}p_0\left(k_u(\bt),\eta_u(\bt^*)\right)\\
			\times\sum_{\sum_{u\in z_n(\bt^*)}t_u=\ell}\binom{\ell}{t_u,u\in z_n(\bt^*)}\frac{\prod_{u\in z_n(\bt^*)}f_{t_u}(m_u(\bt^*),1)\xi_{t_u}}{\xi_\ell f_\ell\left(M_n(\bt^*),Z_n(\bt^*)\right)}\\
			\times\prod_{u\in z_n(\bt^*)}\frac{f_{t_u}(m_u(\bt^*)+\eta_u(\bt^*)),k_u(\bt^*))}{f_{t_u}(m_u(\bt^*),1)}\cdot
\end{multline*}
	Using Formula \eqref{toto} yields
\begin{multline*}
P(r^{*}_{n+1}(\tau_{\ell})=r^{*}_{n+1}(\bt^{*}))\\
=\frac{1}{\xi_\ell}\p(r_n^*(\tau^*)=r_n^*(\bt^*))\prod_{u\in z_n(\bt^*)}p_0\left(k_u(\bt),\eta_u(\bt^*)\right)\\
				\times\sum_{\sum_{u\in z_n(\bt^*)}t_u=\ell}\binom{\ell}{t_u,u\in z_n(\bt^*)}\prod_{u\in z_n(\bt^*)}f_{t_u}(m_u(\bt^*)+\eta(\bt^*),k_u(\bt))\xi_{t_u}.
\end{multline*}
Finally, Equation \eqref{def_p} and Lemma \ref{relf_k_n+1} and then Equation \eqref{toto} give 
\begin{align*}
P(r^{*}_{n+1}(\tau_{\ell})=r^{*}_{n+1}(\bt^{*})) & =\p(r^{*}_{n+1}(\tau^{*})=r^{*}_{n+1}(\bt^{*}))f_\ell(M_{n+1}(\bt^*),Z_{n+1}(\bt^*))\\
&=\Q_\ell(r^{*}_{n+1}(\tau^{*})=r^{*}_{n+1}(\bt^{*})),
\end{align*}
	which gives the result for $n+1$.
\end{proof}

\begin{rqqq}
Note that under $\mathbb Q_\ell$, the extinction probability of the population is equal to 1. Indeed, using the definition of $\mathbb Q_\ell$, the Lebesgue’s dominated convergence theorem and the fact that $Z_{\infty}=0$ $\p$-a.s.:
\begin{align*}
\lim_{n\rightarrow+\infty}\mathbb Q_{\ell}(Z_n=0)=\lim_{n\rightarrow+\infty}\mathbb \E[\mathds{1}_{Z_n=0}f_\ell(M_n,0)]\\
=\frac{1}{\xi_\ell}\lim_{n\rightarrow+\infty}\mathbb \E[\mathds{1}_{Z_n=0}M_n^\ell]=\frac{\E[M^\ell_{\infty}]}{\xi_\ell}=1.
\end{align*}  
Moreover under $\mathbb Q_\ell$, one can notice that for all $n\in\mathbb N^*$, $Z_n$ does not necessary admit a first moment. Just assume that $Z_1$ (and so $Z_n$) does not admit a moment of order $\ell+1$ under $\p$, as:
\begin{align*}
\xi_\ell\E^{\mathbb Q_\ell}[Z_n]&=\xi_\ell\E[Z_nf(M_n,Z_n)]\\
&\ge \E\left[Z_n\sum_{\underset{t_i\le1,\forall i}{t_1+\ldots+t_{Z_n}=\ell}}\binom{\ell}{t_1\cdots t_{Z_n}}\prod_{i=1}^{Z_n}\xi_{t_i}\right]=\ell!\xi_1^\ell\E\left[Z_n\binom{Z_n}{Z_n-\ell}\right].
\end{align*}
\end{rqqq}

\subsection{The super-critical case $\mu > 1$}

Let us begin with giving an equivalent of $\E[M_p^\ell]$ as $p\to +\infty$.

\begin{prop}\label{expM_n^k}
Let $\ell\in\N^*$.	Let $\bp$ be an offspring distribution which admits a moment of order $\ell\in\N^*$ and such that $\mu>1$, and let $\bq$ be a mark function satisfying \eqref{condq}. Then, there exists a strictly positive constant $\omega_\ell$ such that:
	\begin{equation}\label{espMn^k}
		\E[M_p^\ell]\underset{p\to\infty}{\sim}\mu^{\ell p}\omega_\ell.
	\end{equation}
\end{prop}

\begin{proof}
We prove this proposition by induction on $\ell$.

We have for all $p,n\in \N^*$, using \eqref{expM_p}:
\begin{equation}\label{expW_n}
	\E\left[M_{n+p}|\cF_n\right]=M_n+Z_n\E[M_{p}],
\end{equation}
and an induction reasoning gives
	\begin{equation}\label{espMn}
		\E[M_p]=\left\{
		\begin{array}{ll}
			\E[M_1]\frac{1-\mu^p}{1-\mu},& \text{ if } \mu\neq1,\\
			p\E[M_1],&\text{ otherwise.}
		\end{array}
		\right. 
	\end{equation}
We deduce that \eqref{espMn^k} is true for $\ell=1$,
with $\omega_1:=\frac{\E[M_1]}{\mu-1}$ when $\mu> 1$. 

Now, assume that \eqref{espMn^k} is true for all $1\le i<\ell$ for some $\ell>1$, in other words:
\begin{equation}\label{limit1}
	\frac{\E\left[M_p^i\right]}{\mu^{pi}}= 
	\omega_i(1+o(1)).
	\end{equation}
In order to obtain the result at rank $\ell$, we write:
\begin{equation*}
	\E[M_{p}^\ell]=\mu^{p}\sum^{p-1}_{m=0}\left(\frac{\E[M_{m+1}^\ell]}{\mu^{m+1}}-\frac{\E[M_{m}^\ell]}{\mu^{m}}\right)=:\mu^{p}\sum_{m=0}^{p-1}\frac{\Delta_m}{\mu^{m+1}},
\end{equation*}
and we study the telescopic series with general term $\frac{\Delta_{m}}{\mu^{m+1}}$. 
To begin with, taking the expectation in \eqref{expM_p^k} with $n=1$ and $p=m$ (and using again that $M_1^j=M_1$ for every $j\ge 1$), we obtain by isolating the case $j=\ell$:
\begin{align}\label{expexpM_n^k}
	\frac{\Delta_m}{\mu^{m\ell}}=&\sum_{j=0}^{\ell-1}\binom{\ell}{j}\E\left[\frac{M_1}{\mu^{m(\ell-j)}}\sum_{t_1+\dots+t_{Z_1}=j}\binom{j}{t_1\cdots t_{Z_1}}\prod_{i=1}^{Z_1}\frac{\E[M_m^{t_i}]}{\mu^{mt_i}}\right]\nonumber\\
	&+\E\left[\sum_{\underset{t_i<\ell,\forall i}{t_1+\cdots+ t_{Z_1}}=\ell}\binom{\ell}{t_1\cdots t_{Z_1}}\prod_{i=1}^{Z_1}\frac{\E[M_m^{t_i}]}{\mu^{mt_i}}\right]\nonumber\\
	&:=\sum_{j=0}^{\ell-1}\binom{m}{j}\E[A_{j,m}]+\E[B_m]
\end{align}
with $B_m=0$ if $Z_1\le 1$.

Note that in $B_m$, each term of the sum has at most $\ell$ strictly positive $t_i$. Thus, setting $\omega_0=1$ and $\varpi_\ell:=\max_{0\leq j <\ell}\omega_j$ and using \eqref{limit1}  for $m$ large enough:
\begin{align}\label{upperbound}
	0\le B_m\leq \left(\frac{3}{2}\right)^\ell\varpi_\ell^\ell \sum_{\underset{t_i<\ell,\forall i}{t_1+\dots+t_{Z_1}=\ell}}\binom{\ell}{t_1\cdots t_{Z_1}}\leq\left(\frac{3}{2}\varpi_\ell Z_1\right)^\ell\in L^1.
\end{align}
Using \eqref{limit1} and \eqref{upperbound}, Lebesgue's dominated convergence theorem ensures that:
\begin{align*}
	\lim_{m\rightarrow+\infty}\E[B_m]=\E\left[\lim_{m\rightarrow+\infty}B_m\right]=\E\left[\sum_{\underset{t_i<\ell,\forall i}{t_1+\cdots+t_{Z_1}=\ell}}\binom{\ell}{t_1\cdots t_{Z_1}}\prod_{i=1}^{Z_1}\omega_{t_i}\right]=:\tilde c_\ell.
\end{align*}
A similar reasoning gives that:
\begin{multline*}
	\lim_{m\rightarrow+\infty}\sum_{j=0}^{\ell-1}\binom{\ell}{j}\E[A_{j,m}]\\
	=\lim_{m\rightarrow+\infty}\sum_{j=0}^{\ell-1}\binom{\ell}{j}\frac{1}{\mu^{m(\ell-j)}}\E\left[M_1\sum_{t_1+\ldots+t_{Z_1}=j}\binom{j}{t_1\cdots t_{Z_1}}
	\prod_{i=1}^{Z_1}\omega_{t_i}\right]=0, 
\end{multline*}
implying that $\Delta_m\sim {\mu^{m\ell}}\tilde c_\ell$ when $m$ goes to infinity. As a result: 

\begin{multline*}
\E[M_{p}^\ell]=\mu^{p}\sum_{m=0}^{p-1}\frac{\Delta_m}{\mu^{m+1}}\underset{p\to\infty}{\sim}\mu^{p-1}\sum^{p-1}_{m=0}\mu^{m(\ell-1)}\tilde{c}_\ell\\
\underset{p\to\infty}{\sim}\mu^{\ell p}\tilde{c}_\ell(\mu^{\ell}-\mu)^{-1}=:\mu^{\ell p}\omega_\ell.
\end{multline*}
\end{proof}

We deduce the limit of Theorem \ref{thm:limit_polynome} in the super-critical case.

\begin{theo}\label{limmart^ksurmart}
	Let $\bp$ be an offspring distribution which admits a moment of order $\ell\in\NN$ such that $\mu>1$ and let $\bq$ be a mark function satisfying \eqref{condq}. For every $n\in\N$ and every $\Lambda_n\in\cF_n$, we have:
	\begin{equation}
		\lim_{p\to +\infty}\frac{\E\left[\indic_{\Lambda_n}M_{n+p}^\ell\right]}{\E\left[M_{n+p}^\ell\right]}=\E\left[\indic_{\Lambda_n}\frac{P_\ell(Z_n)}{\mu^{\ell n}}\right],
	\end{equation}
	where $P_\ell$ is a polynomial function of degree $\ell$.
\end{theo}

\begin{proof}
Let $p,n\in\N$ and note that  \eqref{expM_p^k} is equivalent to:
\begin{align*}
	\frac{\E\left[\left.M_{n+p}^\ell\right|\cF_n\right]}{\mu^{\ell (n+p)}}
	&=\sum_{j=0}^\ell\binom{\ell}{j}\frac{M_n^{\ell-j}}{\mu^{(\ell-j)p+n\ell}}{\sum_{t_1+\ldots+t_{Z_n}=j}}\binom{j}{t_1\cdots t_{Z_n}}
	\prod_{i=1}^{Z_n}\frac{\E\left[M_{p}^{t_i}\right]}{\mu^{pt_i}}\\
	&:=\sum_{j=0}^\ell\binom{\ell}{j}D_{j,p}.
\end{align*}
With the same notation and reasoning as those in the proof of Proposition \ref{expM_n^k}, for all $j\in\llbracket0,\ell\rrbracket$:
\begin{align}\label{upperbound2}
	D_{j,p}\le M_n^{\ell-j}\left(\frac{3}{2}\varpi_\ell Z_n\right)^j\le C \tilde Z_n^\ell,
\end{align}
where $C$ is a positive constant and $\tilde Z_n:={\sum_{i=0}^n}Z_i$, the total population up to generation $n$, belongs to $L^\ell$ according to Lemma \ref{propmomordk}. Remark that we use that $M_n\le \tilde Z_n$.

Thanks to Proposition \ref{expM_n^k}, almost surely:
\begin{align}\label{almostsurely}
	\lim_{p\rightarrow+\infty}\sum_{j=0}^\ell\binom{\ell}{j}D_{j,p}&=\lim_{p\rightarrow+\infty}\sum_{j=0}^\ell\binom{\ell}{j}\frac{M_n^{\ell-j}}{\mu^{(\ell-j)p+n\ell}}{\sum_{t_1+\ldots+t_{Z_n}=j}}\binom{j}{t_1\cdots t_{Z_n}}\prod_{i=1}^{Z_n}\omega_{t_i}\nonumber\\
	&=\frac{1}{\mu^{n\ell}}{\sum_{t_1+\ldots+t_{Z_n}=\ell}}\binom{\ell}{t_1\cdots t_{Z_n}}
	\prod_{i=1}^{Z_n}\omega_{t_i}.
\end{align}
Using \eqref{upperbound2}, \eqref{almostsurely} and  Proposition \ref{expM_n^k}, Lebesgue's dominated convergence theorem gives:
\begin{align*}
	\lim_{p\rightarrow+\infty}\frac{\E\left[\mathds{1}_{\Lambda_n}M_{n+p}^\ell\right]}{\E\left[M_{n+p}^\ell\right]}
	=\E\left[\mathds{1}_{\Lambda_n}\frac{1}{\mu^{n\ell}\omega_\ell}\sum_{t_1+\ldots+t_{Z_n}=\ell}\binom{\ell}{t_1\cdots t_{Z_n}}
	\prod_{i=1}^{Z_n}\omega_{t_i}\right].
\end{align*}

To prove that the sum in the right hand side of the above equation is indeed a polynomial function in $Z_n$,  let us introduce, for $1\le i\le \ell$ the set
\begin{equation}\label{expSik}
S_{q,\ell}=\left\{\bn=(n_1,\ldots ,n_q)\in(\N^*)^q, \sum_{m=1}^q n_m=\ell\right\}
\end{equation}
and for $\bn\in S_{q,\ell}$, $z\in\N^*$,
\[
A_{\bn, z}=\{(t_1,\ldots,t_z)\in \N^z,\ \{t_1,\ldots,t_z\}=\{n_1,\ldots,n_q,0\}\}.
\]
Then, we have
\begin{align*}
\sum_{t_1+\ldots+t_{Z_n}=\ell} & \binom{\ell}{t_1\cdots t_{Z_n}}
	\prod_{i=1}^{Z_n}\omega_{t_i}\\
	 & =\sum_{q=1}^\ell \sum _{\bn\in S_{q,\ell}}\sum _{(t_1,\ldots,t_{Z_n})\in A_{\bn,Z_n}}\binom{\ell}{t_1\cdots t_{Z_n}}
	\prod_{i=1}^{Z_n}\omega_{t_i}\\
	& =\sum_{q=1}^\ell\sum _{\bn\in S_{q,\ell}}\binom{\ell}{n_1\cdots n_q}\left(\prod_{i=1}^q\omega_{n_i}\right)\mathrm{Card}(A_{\bn,Z_n})
\end{align*}
and $\mathrm{Card}(A_{\bn,Z_n})$ is clearly a polynomial of degree $q$ in $Z_n$ as it can be written $\binom{Z_n}{q}b_\bn$ where $b_\bn$ is a constant depending only of $\bn$.
\end{proof}

\begin{rqqq}
The martingale $P_\ell(Z_n)/{\mu^{\ell n}}$ already appears in \cite{abraham_penalization_2020}. By uniqueness property of Proposition 4.3 of \cite{abraham_penalization_2020}, this is indeed the same martingale.
\end{rqqq}

\begin{rqqq}
The (unmarked) random tree whose distribution is given by the change of probability using the martingale $P_\ell(Z_n)/{\mu^{\ell n}}$ is already described in Definition $4.4$ of \cite{abraham_penalization_2020}. As this martingale does not depends on the marks, it is not difficult to see that the random marked tree associated (via the Girsanov transformation) with this martingale has the same mark function $\bq$ that is:
\begin{itemize}
\item The unmarked tree is those of  Definition $4.4$ of \cite{abraham_penalization_2020},
\item Conditionally on this tree, all the nodes are marked independently of each others and a node $u$ with $k_u$ offspring has mark one with probability $\bq(k_u)$.
\end{itemize}
\end{rqqq}

\subsection{The critical case $\mu=1$}

As for the super-critical case, we need an equivalent of $\E[M_p^\ell]$ as $p\to+\infty$.

\begin{prop}\label{expM_n^kmu1}
Let $\bp$ be an offspring distribution satisfying \eqref{condp}, that admits a moment of order $\ell\in\NN$ and satisfies $\mu=1$. Let $\bq$ be a mark function satisfying \eqref{condq}. Then, there exists a positive constant $\tilde{\omega}_\ell$ such that:
	\begin{equation}\label{eqMn^kmu1}
		\E[M_p^\ell]
		\underset{p\to\infty}{\sim}p^{2\ell-1}\tilde{\omega}_\ell.
	\end{equation}
\end{prop}

\begin{proof}
The proof is similar to the proof of Proposition \ref{expM_n^k}.
According to \eqref{espMn}, \eqref{eqMn^kmu1} is true for $\ell=1$ with $\tilde{\omega}_{1}:=\E[M_1]$.
Now, assume that there exists $\ell$ such that \eqref{eqMn^kmu1} is true for all $1\leq i<\ell$, in other words:
\begin{equation}\label{limit2}
	\E\left[M_p^i\right]= 
		p^{2i-1}\tilde{\omega}_i(1+o(1)).
\end{equation}
In order to obtain the equivalent, we write:
\begin{equation*}
	\E\left[M_p^\ell\right]=\sum^{p-1}_{i=0}\left(\E[M_{i+1}^\ell]-\E[M_{i}^\ell]\right)=\sum_{i=0}^{p-1}\Delta_i,
\end{equation*}
and we study the telescopic series with general term $\Delta_i$. 
According to Proposition \ref{expM_n^k}, more precisely \eqref{expexpM_n^k}, we obtain:
\begin{align*}
	\frac{\Delta_m}{m^{2(\ell-1)}}
	&=\sum_{j=0}^{\ell-1}\frac{\E\left[A_{j,m}\right]}{m^{2(\ell-1)}}+\frac{\E\left[B_{m}\right]}{m^{2(\ell-1)}},
\end{align*}
with $\Delta_m$, $A_{j,m}$ and $B_m$ defined in the proof of Proposition \ref{expM_n^k} with $\mu=1$ in our case.

Let $z>1$. We set $\tilde\omega_0=1$ and $\tilde{\varpi}_\ell:=\max_{0\leq j < \ell}\tilde{\omega}_j$ and, for $t=(t_1,\ldots,t_z)$ such that $\forall i,\ 0\le t_i<\ell$ and $\sum _{i=1}^zt_i=\ell$, we set  $N(t):={\sum_{i=1}^{z}\mathds{1}_{t_i>0}}$ the number of positive $t_i$.
Thus, noting that $N(t)\ge 2$ for every term of the sum in $B_m$, and using \eqref{limit2} for $m$ large enough:
\begin{align}\label{ineq}
	0\leq \frac{B_m}{m^{2(\ell-1)}} &=\sum_{\underset{t_i<\ell,\forall i}{t_1+\ldots+t_{Z_1}=\ell}}\binom
		{\ell}
		{t_1\cdots t_{Z_1}}
\frac{m^{2\ell-N(t)}}{m^{2(\ell-1)}}\prod^{Z_1}_{i=1}\tilde{\omega}_{t_i}(1+o(1))\nonumber\\ 
&\leq
 \left(\frac{3}{2}\tilde{\varpi}_m\right)^\ell\sum_{t_1+\ldots+t_{Z_1}=\ell}\binom{\ell}{t_1\cdots t_{Z_1}}
\leq \left(\frac{3}{2}\tilde{\varpi}_m Z_1\right)^\ell\in \mathrm{L}^1.
\end{align}
Note that the only terms in $B_m$ that do not tend to 0 a.s. is when $N(t)=2$. Then using \eqref{limit2} and \eqref{ineq}, Lebesgue's dominated convergence theorem ensures that:
\begin{multline}\label{limitB_p2}
	\lim_{m\rightarrow+\infty}\frac{\E[B_m]}{m^{2(\ell-1)}}=\E\left[\lim_{m\rightarrow+\infty}\frac{B_m}{m^{2(\ell-1)}}\right]\\
	=\E\left[\frac{Z_1(Z_1-1)}{2}\right]\sum_{\underset{t_1,t_2>0}{t_1+t_{2}=\ell}}\binom{\ell}{t_1~ t_2}\tilde{\omega}_{t_1}\tilde{\omega}_{t_2}=:\tilde d_\ell.
\end{multline}
A similar reasoning gives that $\lim_{m\rightarrow+\infty}\frac{\E[A_{j,m}]}{m^{2(\ell-1)}}=0$ and we can conclude with \eqref{limitB_p2} that $\Delta_m\sim {m^{2(\ell-1)}}\tilde d_\ell$ when $m$ goes to infinity. As a result: 
\begin{equation*}
		\E\left[M_p^\ell\right]=\sum^{p-1}_{i=0}\Delta_i\underset{p\to+\infty}{\sim} \tilde{d}_\ell\sum^{p-1}_{i=0}i^{2(\ell-1)}.
\end{equation*}
We then use the following property
\begin{equation*}
	\sum^{p-1}_{i=0}i^m=\frac{1}{m+1}\sum_{i=0}^m(-1)^i\binom{m+1}{i}\mathscr B_i(p-1)^{m+1-i}\underset{p\to +\infty}{\sim}\frac{p^{m+1}}{m+1},
\end{equation*}
where $\mathscr B_j$ are Bernoulli numbers, see \cite{knuth_johann_1993}, and we get
\[
\E\left[M_p^\ell\right]\underset{p\to+\infty}{\sim}\frac{p^{2\ell-1}}{2\ell-1}\tilde{d}_\ell=:p^{2\ell-1}\tilde{\omega}_\ell.
\]
\end{proof}

We deduce Theorem \ref{thm:limit_polynome} for the case $\mu=1$:
\begin{theo}\label{limmart^kmart}
	Let $\bp$ be an offspring distribution satisfying \eqref{condp} that admits a moment of order $\ell\in\NN$ and such that $\mu=1$ and let $\bq$ be a mark function satisfying \eqref{condq}. For every $n\in\N$ and for every $\Lambda_n\in\cF_n$, we have:
	\begin{equation}\label{limitmart^kmu1}
		\lim_{p\to+\infty}\frac{\E\left[\indic_{\Lambda_n}M_{n+p}^\ell\right]}{\E\left[M_{n+p}^\ell\right]}=\E[\indic_{\Lambda_n}Z_n].
	\end{equation}
\end{theo}
\begin{proof}Let $p,n\in\N$, here we write: 
\begin{align*}
	&\frac{\E[M_{n+p}^\ell|\cF_n]}{p^{2\ell-1}}\\
	&=\sum_{j=0}^\ell\binom{\ell}{j}M_n^{\ell-j}{\sum_{t_1+\ldots+t_{Z_n}=j}}\binom{j}{t_1\cdots t_{Z_n}}\frac{1}{p^{2(\ell-j)-1+N(t)}}
	\frac{\prod_{\ell=1}^{Z_n}\E\left[M_{p}^{t_{\ell}}\right]}{p^{2j-N(t)}}\\
	&=:\sum_{j=0}^\ell\binom{\ell}{j}M_n^{\ell-j}{\sum_{t_1+\ldots+t_{Z_n}=j}}\binom{j}{t_1\cdots t_{Z_n}}W_{j,N(t),p},
\end{align*}
and \eqref{upperbound2} remains true replacing $\mu^{(\ell-j)p+nj}$ by $p^{2(\ell-j)-1+N(t)}$ giving us an upper bound independent of $p$.\\
 Note that $W_{j,N(t),p}$ tends to 0 a.s. except for $j=\ell$ and $N(t)=1$ implying that, using \eqref{eqMn^kmu1}:
 \begin{align*}
 \lim_{p\rightarrow+\infty}\frac{\E\left[\left.M_{n+p}^\ell\right|\cF_n\right]}{p^{2\ell-1}}= \lim_{p\rightarrow+\infty}\frac{Z_n\E\left[M_{p}^\ell\right]}{p^{2\ell-1}}=Z_n\tilde\omega_\ell.
 \end{align*}
 To conclude, as in Theorem \ref{limmart^ksurmart}, we use \eqref{eqMn^kmu1} again and Lebesgue’s dominated convergence theorem to obtain \eqref{limitmart^kmu1}.
\end{proof}

\begin{rqqq}
As for the super-critical case, the martingale is a known one and the associated random marked tree is a size-biased tree (i.e. Kesten's tree, see \cite{kesten_subdiffusive_1986}) associated with the offspring distribution $\bp$ and  with the mark function $\bq$ being unchanged.
\end{rqqq}

\section{Exponential penalization}\label{sec:expo}
We now consider a penalization of the form $s^{M_n}t^{Z_n}$ with $s\in(0,1)$, $t\in[0,1]$.
The case $s=1$ has already beed studied in \cite{abraham_penalization_2020}.

\medskip
To begin with, let us set for $p\ge 1$, $s,t\in[0,1]$, $\varphi_p(s,t)=\E[s^{M_p}t^{Z_p}]$. For $s\in (0,1)$ fixed, we consider the function $f_s$ defined for $t\in [0,1]$ by
\[
f_s(t)=\varphi_1(s,t)=\sum_{k=0}^{+\infty}\bp(k)t^k(s\bq(k)+1-\bq(k)).
\]
As $f_s$ is $\mathscr{C}^\infty$ on $[0,1)$, increasing, convex and satisfies $0\le f_s\le 1$, $f_s(1)<1$, the function $f_s$ admits a unique fixed point denoted $\kappa(s)$ on $[0,1]$. Furthermore, using the notation $f^p=f\circ\ldots\circ f$ for the $p$-th composition of $f$, we have $\varphi_p(s,t)=f^p_s(t)$
and
\begin{equation}\label{limespsMptZp}
\forall t\in[0,1],\,\lim_{p\rightarrow+\infty}\E[s^{M_p}t^{Z_p}]=\lim_{p\rightarrow+\infty}f_s^p(t)=\kappa(s).
\end{equation}

 Note that $\kappa(s)=0$ if and only if $\bp(0)=0$. Therefore, we must study separately the two cases $\bp(0)>0$ and $\bp(0)=0$.

\subsection{Case $\bp(0)>0$}

\begin{theo}\label{pensMntZn1}
	\begin{enumerate}
	\item Let $\bp$ be an offspring distribution satisfying \eqref{condp} and such that $\bp(0)>0$, and let $\bq$ be a mark function satisfying \eqref{condq}. 
For every $t\in[0,1]$ and $s\in(0,1)$, for every $n\in\N$, $\Lambda_n\in\cF_n$, we have:
	\begin{equation}
		\frac{\E\left[\indic_{\Lambda_n}s^{M_{n+p}}t^{Z_{n+p}}\right]}{\E\left[s^{M_{n+p}}t^{Z_{n+p}}\right]}\underset{p\to +\infty}{\longrightarrow}\E[\indic_{\Lambda_n}s^{M_n}\kappa(s)^{Z_n-1}].
	\end{equation}
\item The probability measure $\Q_{s,1}$ defined on $(\T^*,\cF_\infty)$ by
\[\forall n\in\mathbb N, \Lambda_n\in\cF_n,\quad\Q_{s,1}(\Lambda_n)=\E\left[\indic_{\Lambda_n}s^{M_n}\kappa(s)^{Z_n-1}\right],\]
is the distribution of a MGW tree with reproduction-marking law $p_{s,0}$ defined by : 
\begin{equation}\label{dfps0}
\forall k\in\mathbb N,\ \forall \eta\in\lbrace0,1\rbrace,\quad  p_{s,0}(k,\eta)=\bp_0(k,\eta)s^\eta\kappa(s)^{k-1}
\end{equation}
where $\bp_0(k,\eta)$ is defined by \eqref{def_p0}.
\end{enumerate}
\end{theo}

\begin{proof}
\begin{enumerate}
\item Thanks to the branching property and \eqref{expM_p}, for $p,n\in\N$, we have:
\begin{equation}\label{expespsMptZp}
	\E\left[s^{M_{n+p}}t^{Z_{n+p}}|\cF_n\right]
	=s^{M_n}	\E\left[s^{M_{p}}t^{Z_{p}}\right]^{Z_n}=s^{M_n}f_s^{p}(t)^{Z_n}.
\end{equation}
As $s^{M_n}f_s^{p}(t)^{Z_n}\leq 1$, by dominated convergence theorem and according to \eqref{limespsMptZp}, we obtain:
\begin{align*}
	\lim_{p\to+\infty}\frac{\E\left[\indic_{\Lambda_n}s^{M_{n+p}}t^{Z_{n+p}}\right]}{\E\left[s^{M_{n+p}}t^{Z_{n+p}}\right]} & =\lim_{p\to+\infty}\frac{\E\left[\indic_{\Lambda_n}	s^{M_n}f_s^{p}(t)^{Z_n}\right]}{f_s^{n+p}(t)}\\
	&=\E\left[\indic_{\Lambda_n}	s^{M_n}\kappa(s)^{Z_n-1}\right].
\end{align*}
\item Let $\tau_{s,1}$ be a MGW tree with the reproduction-marking law $p_{s,0}$ defined on some probability space $(\Omega,\mathscr{A},P)$. To obtain our result, it suffices to prove that:
\[\forall \bt^{*}\in\T^{*}, \forall n\in \mathbb N,\quad \Q_{s,1}(r^{*}_{n}(\tau^{*})=r^{*}_{n}(\bt^{*}))=P(r^{*}_{n}(\tau_{s,1})=r^{*}_{n}(\bt^{*})).\]
Let $\bt^*=(\bt, (\eta_u)_{u\in\bt})\in\T^*$. As $\sum_{u\in r_{n-1}(\bt)}(k_u(\bt)-1)=Z_n(\bt^*)-1$, we obtain our result since:
\begin{align*}
P & (r^{*}_{n}(\tau_{s,1})=r^{*}_{n}(\bt^{*}))\\
& =\prod_{u\in r^{*}_{n-1}(\bt^{*})}p_{s,0}(k_u(\bt),\eta_u)\\
&=\prod_{u\in r^{*}_{n-1}(\bt^{*})}\bp_0(k_u(\bt),\eta_u)\kappa(s)^{k_u(\bt)-1}s^{\eta_u}\\
&=s^{M_n(\bt^*)}\kappa(s)^{\sum_{u\in r_{n-1}(\bt)}k_u-1}\prod_{u\in r^{*}_{n-1}(\bt^{*})}\bp_0(k_u(\bt),\eta_u)\\
&=s^{M_n(\bt^*)}\kappa(s)^{Z_n(\bt^*)-1}\p(r^{*}_{n}(\tau^*)=r^{*}_{n}(\bt^{*}))\\
&=\Q_{s,1}(r^{*}_{n}(\tau^{*})=r^{*}_{n}(\bt^{*})).
\end{align*}
\end{enumerate}

\end{proof}

As in \cite{abraham_penalization_2020}, the previous penalization can be generalized by looking at the penalization $H_\ell(Z_p)s^{M_p}t^{Z_p}$ where $H_\ell$ is the $\ell$-th Hilbert polynomial defined for $\ell\ge 1$ by
\[
H_\ell(x)=\frac{1}{\ell!}\prod_{j=0}^{\ell-1} (x-j).
\]

Recall the definition of the set $S_{q,\ell}$ of \eqref{expSik}.
Using Faà di Bruno's formula, we have for any functions $f,h$ of class $\mathcal{C}^\ell$:
\begin{align}\label{faadibrunoform}
	\left(f\circ h\right)^{(\ell)}(s)
&=\sum_{i=1}^\ell \frac{\ell!}{i!}f^{(i)}(h(s))\sum_{(n_1,\ldots,n_i)\in S_{i,\ell}}\prod_{j=1}^i\frac{h^{(n_j)}(s)}{n_j!}\cdot
\end{align}

We first need asymptotics for the derivatives of the $p$-th composition of $f_s$ as $p\to +\infty$.

\begin{lem}\label{equivf_sp^k}
	Let $\bp$ be an offspring distribution  satisfying \eqref{condp}, $\bp(0)>0$ and that admits a moment or order $\ell\in\NN$, and let $\bq$ be a mark function satisfying \eqref{condq}. Then, for all $s\in (0,1)$, there exists a positive function $C_\ell$ such that for all $t\in[0,1]$:
	\begin{equation}\label{expequivf_sp^k}
		(f_s^p)^{(\ell)}(t)\underset{p\to+\infty}{\sim} f_s'(\kappa(s))^pC_\ell(t).
	\end{equation} 
\end{lem}

 {\begin{proof}
 Note that $f_s'(\kappa(s))\in(0,1)$ since $f_s$ is convex, increasing on $[0,1]$, $f_s(0)>0$ and $f_s(1)<1$ by \eqref{condq}.
 
 To prove \eqref{expequivf_sp^k} for $\ell=1$, we use a proof similar to that in \cite{athreya_branching_2004} p.38.	We study the function 
	\begin{equation}\label{dfQp(t)}
		Q_p(t):=f_s'(\kappa(s))^{-p}\left(f_s^p(t)-\kappa(s)\right).
	\end{equation}
	This function can be derived and we have:
	\begin{equation*}
		Q_p'(t)=\frac{(f_s^p)'(t)}{f_s'(\kappa(s))^{p}}=\prod^{p-1}_{j=0}\frac{f_s'(f_s^j(t))}{f_s'(\kappa(s))}=\prod^{p-1}_{j=0}\left(1+\left(\frac{f_s'(f_s^j(t))}{f_s'(\kappa(s))}-1\right)\right).
	\end{equation*}
	We fix $t\in[0,1]$ and show first that $Q_p'(t)$ admits a positive limit $Q^\prime(t)$ when $p$ goes to infinity. The positivity is clear as $f_s^j(t)\in[0,1[$ implying that $f_s'(f_s^j(t))>0$.\\
The existence of this limit is equivalent to the convergence of the series with general term  $a_j:=| f_s'(\kappa(s))-f_s'(f_s^j(t))|$.\\
Let $\varepsilon>0$ be such that $0<\kappa(s)+\varepsilon<1$ and let $\gamma:=f_s'(\kappa(s)+\varepsilon)<1$. According to \eqref{limespsMptZp}, there exists $j_0$ such that for all $j\geq 0$, $f_s^{j+j_0}(t)<\kappa(s)+\varepsilon$. 
	Let $j\in\N$. Since $f_s''$ is increasing and $\max(\kappa(s),f_s^{j+j_0}(t))<\kappa(s)+\varepsilon$, thanks to the mean value inequality, we have: 
\begin{equation*}
	a_{j+j_0}=|f_s'(\kappa(s))-f_s'(f_s^{j+j_0}(t))|\leq f_s''(\kappa(s)+\varepsilon)| \kappa(s)-f_s^{j+j_0}(t)|.
\end{equation*}
With the same arguments applied to $f_s'$ and using the fact that $f_s(\kappa(s))=\kappa(s)$:
\begin{equation*}
	|\kappa(s)-f_s^{j+j_0}(t)|\leq \gamma|f_s^{j+j_0-1}(\kappa(s))-f_s^{j+j_0-1}(t)|.
\end{equation*}
By iteration we obtain: $
	a_{j+j_0}\leq  f_s''(\kappa(s)+\varepsilon)\gamma^j,$
since $0<\gamma<1$, we have $\sum^{\infty}_{j=0}a_j<+\infty$.

Therefore, $\frac{(f_s^p)'(t)}{f_s'(\kappa(s))^{p}}\underset{p\to +\infty}{\longrightarrow}Q'(t)$, implying \eqref{expequivf_sp^k} for $\ell=1$ with $C_1(t)=Q'(t)$. 

 The proof is similar of the induction step of Lemma 3.2 in \cite{abraham_penalization_2020}. We assume that \eqref{expequivf_sp^k} is true for all $j\leq \ell-1$. We use Faà di Bruno's Formula \eqref{faadibrunoform}, and we obtain:
\begin{multline}\label{exptelf_spk}
		\frac{(f_s^{p+1})^{(\ell)}(t)}{(f_s^{p+1})'(t)}-	\frac{(f_s^{p})^{(\ell)}(t)}{(f_s^{p})'(t)}\\
		=\frac{1}{(f_s^{p+1})'(t)}\sum^{\ell}_{i=2}\frac{\ell!}{i!}f_s^{(i)}(f_s^{p}(t))\sum_{(n_1,\ldots,n_i)\in S_{i,\ell}}\prod_{j=1}^i\frac{\left(f_s^{p}\right)^{(n_j)}(t)}{n_j!}
\end{multline}
Since for every $2\leq i\leq \ell$ and every $(n_1,\ldots,n_i)\in S_{i,\ell}$, $n_j<\ell$ for all $j\leq i$, we can use the induction hypothesis:
\begin{multline}\label{equivtel}
	\sum_{(n_1,\ldots,n_i)\in S_{i,\ell}}\prod_{j=1}^i\frac{\left(f_s^{p}\right)^{(n_j)}(t)}{n_j!}\\
	\underset{p\to+\infty}{\sim}\sum_{(n_1,\ldots,n_i)\in S_{i,\ell}}\prod_{i=1}^i\frac{f_s'(\kappa(s))^p C_{n_j}(t)}{n_j!}
	=:f_s'(\kappa(s))^{pi}K_i(t),
\end{multline}
with $K_i$ a positive function.

Equation \eqref{limespsMptZp} and the continuity of $f_s^{(j)}$ imply that 
\[\lim_{p\to+\infty}f_s^{(j)}(f_s^{p}(t))=f_s^{(j)}(\kappa(s))
\]
 for all $j\geq 1$.
 
Moreover, since $f_s'(\kappa(s))<1$,  
Equations \eqref{exptelf_spk} and \eqref{equivtel} imply that:
\begin{align*}
	\frac{(f_s^{p+1})^{(\ell)}(t)}{(f_s^{p+1})'(t)}-	\frac{(f_s^{p})^{(\ell)}(t)}{(f_s^{p})'(t)}
		& \underset{p\to+\infty}{\sim}\frac{\ell!}{2!}f_s''(\kappa(s))\frac{f_s'(\kappa(s))^{2p}}{{(f_s^{p+1})'(t)}}K_2(t)\\
	& \underset{p\to+\infty}{\sim}f_s'(\kappa(s))^{p}\tilde{K}(t),
\end{align*}
with $\tilde{K}(t)$ a positive function.

Since $0<f_s'(\kappa(s))<1$, we have on $[0,1]$:
\begin{multline*}
	0<\tilde{C_\ell}(t):=\lim_{p\to+\infty} \frac{\left(f_s^{p}\right)^{(\ell)}(t)}{\left(f_s^{p}\right)'(t)}\\
	=\frac{(f_s)^{(\ell)}(t)}{f_s'(t)}+\sum_{p\geq 1}\left(	\frac{(f_s^{p+1})^{(\ell)}(t)}{(f_s^{p+1})'(t)}-	\frac{(f_s^{p})^{(\ell)}(t)}{(f_s^{p})'(t)}\right)<+\infty.
\end{multline*}
The previous result is equivalent to 
$	\left(f_s^{p}\right)^{(\ell)}(t)\underset{p\to+\infty}{\sim}\tilde{C_\ell}(t)\left(f_s^{p}\right)'(t)$,
and we conclude with the induction hypothesis and $C_\ell(t):=C_1(t)\tilde{C_\ell}(t)$.\end{proof}}

We can now show the results of the limit with the penalization function $H_\ell(Z_p)s^{M_p}t^{Z_p-\ell}$.
\begin{theo}\label{pensMntZnderk}
	Let $\bp$ be an offspring distribution satisfying \eqref{condp} with $\bp(0)>0$ and that admits a moment of order $\ell\in\NN$, and let $\bq$ be a mark function satisfying \eqref{condq}. Let $n\in\N$, $\Lambda_n\in\cF_n$, $t\in[0,1]$ and $s\in(0,1)$ fixed, we have
	\[
		\frac{\E\left[\indic_{\Lambda_n}H_\ell(Z_{n+p})s^{M_{n+p}}t^{Z_{n+p}}\right]}{\E\left[H_\ell(Z_{n+p})s^{M_{n+p}}t^{Z_{n+p}}\right]}\underset{p\to +\infty}{\longrightarrow}\E\left[\indic_{\Lambda_n}\frac{s^{M_n}Z_n \kappa(s)^{Z_n-1}}{f_s'(\kappa(s))^n}\right].
	\]
\end{theo}

\begin{proof} The proof is inspired by those of Theorem 3.3 in \cite{abraham_penalization_2020}.

Let $s\in(0,1)$ be fixed.
Let $p,n\in\N$, according to \eqref{expespsMptZp} and Faà di Bruno's formula \eqref{faadibrunoform}, we have for every $t\in[0,1]$:
\begin{align}\label{expespsMptZpderk}
	\E[H_\ell(Z_p)&s^{M_{n+p}}t^{Z_{n+p}-\ell}|\cF_n]\\
		&=\frac{s^{M_n}}{\ell!}\frac{\partial^\ell}{\partial t^\ell}\left(f_s^{p}(t)^{Z_n}\right)\nonumber\\
	&=s^{M_n}\sum_{i=1}^\ell H_i(Z_n)f_s^{p}(t)^{Z_n-i}\sum_{(n_1,\ldots,n_i)\in S_{i,\ell}}\prod_{j=1}^i\frac{\left(f_s^{p}\right)^{(n_j)}(t)}{n_j!}\nonumber\\
	&=s^{M_n}\sum_{i=1}^\ell H_i(Z_n)f_s^{p}(t)^{Z_n-i}R_i(t).
\end{align}
Applying \eqref{expequivf_sp^k}, as $0<f_s'(\kappa(s))<1$, we obtain for all $i\geq1$ and $t\in[0,1]$, as $p$ goes to $+\infty$,
\begin{equation}\label{a.s.}
	\frac{R_i(t)}{f^{\prime}(\kappa(s))^{p}}
		=c_i(t)f_s'(\kappa(s))^{p(i-1)}+o\left(f_s'(\kappa(s))^{p(i-1)}\right)
\end{equation}
for some strictly positive function $c_i$. Thereby, using \eqref{limespsMptZp} in \eqref{expespsMptZpderk}, we get a.s.:
\begin{align*}
\lim_{p\rightarrow+\infty}\frac{\E\left[H_\ell(Z_{p+n})s^{M_{p+n}}t^{Z_{p+n}-\ell}|\cF_n\right]}{f^{\prime}(\kappa(s))^{p}}=\frac{c_1(t)}{\ell!}s^{M_n}Z_n\kappa(s)^{Z_n-1}.
\end{align*}
Note that \eqref{a.s.} implies the existence of a positive function $C$, such that, for all $t\in[0,1]$: 
\begin{align*}
\max_{1\le i\le \ell}\frac{R_i(t)}{f^{\prime}(\kappa(s))^{p}}\le C(t),
\end{align*}
and since $s<1$, $ f_s^{p}(t)<1$ for all $t\in[0,1]$:
\begin{align}\label{upperbound3}
0\le \frac{\E\left[H_\ell(Z_{p+n})s^{M_{p+n}}t^{Z_{p+n}-\ell}|\cF_n\right]}{f^{\prime}(\kappa(s))^{p}}\le C(t)\sum_{i=1}^\ell H_i(Z_n).
\end{align}
As $H_i(Z_n)$ is integrable, we have by dominated convergence:
\begin{equation}\label{limcool}
\lim_{p\rightarrow+\infty}\frac{\E\left[\mathds{1}_{\Lambda_n}H_\ell(Z_{p+n})s^{M_{p+n}}t^{Z_{p+n}-\ell}\right]}{f^{\prime}(\kappa(s))^{p}}=\E\left[\mathds{1}_{\Lambda_n}s^{M_n}Z_n\kappa(s)^{Z_n-1}\right].
\end{equation}
We obtain our result taking the previous formula for $n=0$ and $\Lambda_n=\Omega$ and taking the ratio of these two quantities. 
\end{proof}

In order to describe the  new probability $\Q_{s,2}$ defined on $(\T^*,\cF_\infty)$ by:
\begin{equation}\label{def_Qs2}
\forall \Lambda_n\in\cF_n,\ 	\Q_{s,2}(\Lambda_n)=\E\left[\indic_{\Lambda_n}\frac{s^{M_n}Z_n \kappa(s)^{Z_n-1}}{f_s'(\kappa(s))^n}\right],
\end{equation}
we need to introduce the following marked random tree:
\begin{df}\label{def_taus2}
A special random marked tree is a multi-type random marked tree $\tau_{s,2}$ described as follows:
\begin{itemize}
\item the nodes are normal or special;
\item the root is special;
\item the reproduction-marking law of a normal node is $p_{s,0}$ (see \eqref{dfps0});
\item the reproduction-marking law of a special node is
\[
p_{s,1}(k,\eta)=p_{s,0}(k,\eta)\frac{k}{f'_s(\kappa(s))};
\]
\item at each generation, conditionnally given the number of individuals at that generation, there is a unique special node chosen uniformly among all the individuals of this generation.
\end{itemize}
\end{df}

Using the same arguments as in the proof of Theorem \ref{probunderQ_krn}, 
we have the following result:
\begin{theo}\label{probunderQ_(s,k)rn}
The probability $\Q_{s,2}$ defined  by \eqref{def_Qs2} is
the distribution of  a special random marked tree defined in Definition \ref{def_taus2}.
\end{theo}

\subsection{Case $\bp(0)=0$} 

Henceforth, we consider the case where $\bp(0)=0$ implying $\kappa(s)=0$ and denote by $r:=\min\{j>0;~\bp(j)>0\}$ and $\alpha_r(s):=\bp(r)(s\bq(r)+1-\bq(r))$.\\ 
We begin to give an equivalent of Lemma \ref{equivf_sp^k} in this case. 
\begin{lem}\label{expfspk0}
	Assume $r\geq2$. For every $\ell\in\N$, every $s\in (0,1)$ and every $t\in(0,1]$, there exists a positive function $K_{s,\ell}(t)$, such that, as $p\to+\infty$,
	\begin{equation}\label{equivdur}
		(f_s^p)^{(\ell)}(t)=K_{s,\ell}(t)r^{p\ell}e^{r^pb(t)}(1+o(1)),
\end{equation}
where
\begin{equation*}
	b(t):= \ln(t)+\sum_{j=0}^{+\infty}r^{-(j+1)}\ln\left(\frac{f_s^{j+1}(t)}{f_s^j(t)^r}\right).
\end{equation*}
\end{lem}

\begin{proof}The case $\ell=0$ is inspired by the proof of Lemma 10 in \cite{fleischmann_left_2009}. First, note that for all $t\in(0,1]$, 
\[f_s(t)=\E\left[s^{M_1}t^{Z_1}\right]\le \E\left[s^{M_1}\right]=:\gamma(s)\in(0,1). \]
$f_s$ being non decreasing, we have:
\[f_s^2(t)\le f_s(\gamma(s))=\sum_{j\ge r}\gamma(s)^j\alpha_j(s)\le \gamma(s)^r,\]
and with an obvious induction reasoning $f_s^p(t)\le \gamma(s)^{r^{p-1}}$ for all $p\in\mathbb N^*$.

Moreover, as for $0<t\le 1, f_s^p(t)\ne 0$, we can write: 
\begin{align}
\frac{f_s^{p+1}(t)}{\alpha_r(s)f_s^p(t)^r}=1+\sum_{j\ge 1}\frac{\alpha_{j+r}(s)}{\alpha_r(s)}f_s^{p}(t)^j,
\end{align}
implying:
\begin{multline}
0\le \frac{f_s^{p+1}(t)}{\alpha_r(s)f_s^p(t)^r}-1\le  f_s^{p}(t)\sum_{j\ge 1}\frac{\alpha_{j+r}(s)}{\alpha_r(s)}\\
\le \gamma(s)^{r^{p-1}}\sum_{j\ge 1}\frac{\alpha_{j+r}(s)}{\alpha_r(s)}=:\gamma(s)^{r^{p-1}}C_r.
\end{multline}
As $\ln(1+u)\le u$ for every nonnegative $u$, we have:
\[\ln\left(\frac{f_s^{p+1}(t)}{\alpha_r(s)f_s^p(t)}\right)\le \frac{f_s^{p+1}(t)}{\alpha_r(s)f_s^p(t)^r}-1\le \gamma(s)^{r^{p-1}}C_r,\]
and therefore, $b:(0,1]\rightarrow \mathbb R$ is well defined since:
\begin{align*}
-\frac{\ln (\alpha_r(s))}{r-1}+\sum_{p\ge 0}\frac{1}{r^{p+1}}\ln\left(\frac{f_s^{p+1}(t)}{f_s^p(t)^r}\right)=
\sum_{p\ge 0}\frac{1}{r^{p+1}}\ln\left(\frac{f_s^{p+1}(t)}{\alpha_r(s)f_s^p(t)^r}\right)<\infty.
\end{align*}
Using the telescopic property of the series:
\begin{align*}
b(t)-\sum_{j\ge p}\frac{1}{r^{j+1}}\ln\left(\frac{f_s^{j+1}(t)}{f_s^j(t)^r}\right)=\ln t+\sum_{j=0}^{p-1}\frac{1}{r^{j+1}}\ln\left(\frac{f_s^{j+1}(t)}{f_s^j(t)^r}\right)=\frac{1}{r^{p}}\ln (f_s^p(t)).
\end{align*}
and 
\begin{align*}
\sum_{j\ge p}\frac{1}{r^{j+1}}\ln\left(\frac{f_s^{j+1}(t)}{f_s^j(t)^r}\right)&=\ln(\alpha_r(s))\frac{r^{-p}}{r-1}+\sum_{j\ge p}\frac{1}{r^{j+1}}\ln\left(\frac{f_s^{j+1}(t)}{\alpha_r(s)f_s^j(t)^r}\right)
\\&=\ln(\alpha_r(s))\frac{r^{-p}}{r-1}+O\left(\frac{1}{r^p}\gamma(s)^{r^{p-1}}\right).
\end{align*}
As a result: 
\begin{align*}
\ln (f^p_s(t))=r^pb(t)-\frac{\ln(\alpha_r(s))}{r-1}+O\left(\gamma(s)^{r^{p-1}}\right),
\end{align*}
giving:
\begin{equation}\label{expfspk00}
f^p_s(t)=\alpha_r(s)^{\frac{-1}{r-1}}e^{r^pb(t)}(1+o(1))=:K_{s,0}(t)e^{r^pb(t)}(1+o(1)). 
\end{equation}
The result for $\ell=1$ is very similar to the previous one. As $Z_1\in L^1$, for all $s\in(0,1)$, $f_s\in\mathscr C^1((0,1])$, and for $t\in(0,1]$, we have:
\begin{equation*}
	\left(f_s^{p+1}\right)'(t)=\left(f_s^{p}\right)'(t)f_s'(f_s^{p}(t))=\left(f_s^{p}\right)'(t)\sum_{j\ge 0}(r+j)\alpha_{j+r}f_s^{p}(t)^{r+j-1}.
\end{equation*}
This implies with \eqref{expfspk00}:
\begin{align*}\label{ingratk=1}
	0\leq\frac{(f_s^{p+1})'(t)}{(f_s^{p})'(t)r\alpha_r(s) f_s^{p}(t)^{r-1}}-1 & = \sum_{j\ge 1}\frac{(r+j)\alpha_{j+r}(s)}{r\alpha_r(s)}f_s^{p}(t)^{j}\\
	&\leq f_s^{p}(t) \sum_{j\ge1}\frac{(r+j)\alpha_{j+r}(s)}{r\alpha_r(s)}\nonumber\\
&\le \widetilde K_s(t)e^{r^pb(t)},
\end{align*}
where $\tilde K_s$ is a positive function. Using again that, for all $u>0$, $\ln(1+u)<u$, we have:
\begin{align*}
	\sum_{p\ge 0}\ln \left(\frac{(f_s^{p+1})'(t)}{(f_s^{p})'(t)r\alpha_r(s) f_s^{p}(t)^{r-1}}\right)<\infty.
\end{align*}
Thanks to Equation \eqref{expfspk00}, we have:
\begin{equation*}
	\ln \left(\frac{(f_s^{p+1})'(t)}{(f_s^{p})'(t)r\alpha_r(s) f_s^{p}(t)^{r-1}}\right)\underset{p\to+\infty}{\sim}\ln \left(\frac{(f_s^{p+1})'(t)}{(f_s^{p})'(t)re^{r^{p}(r-1)b(t)}}\right), 
\end{equation*}
thereby, it gives: 
\begin{multline*}
	\tilde{b}(t):=\sum_{p\ge0}\ln \left(\frac{(f_s^{p+1})'(t)}{(f_s^{p})'(t)re^{r^{p}(r-1)b(t)}}\right)\\
	=\sum_{p\ge0}\ln \left(\frac{(f_s^{p+1})'(t)r^{-p-1}e^{-r^{p+1}b(t)}}{(f_s^{p})'(t)r^{-p}e^{-r^{p}b(t)}}\right)<\infty.
\end{multline*}
$\tilde b(t)$ being telescopic, we have: 
\begin{equation}
\ln \left(\frac{\left(f_s^{p}\right)'(t)}{r^pe^{(r^{p}-1)b(t)}}\right)=\tilde{b}(t)-\sum_{j\ge p }\ln \left(\frac{(f_s^{j+1})'(t)}{(f_s^{j})'(t)re^{r^{j}(r-1)b(t)}}\right)=\tilde{b}(t)+o(1) ,
\end{equation}
and as a result: 
\begin{equation}
\left(f_s^{p}\right)'(t)=r^pe^{(r^{p}-1)b(t)}e^{\tilde{b}(t)}(1+o(1)),
\end{equation}
which is the claimed result for $\ell=1$ with $K_{s,1}:=e^{\tilde b(t)-b(t)}$.\\
We conclude the proof by induction on $\ell$: assume that there exists $\ell\ge2$ such that \eqref{equivdur} is true for $j<\ell$ and recall \eqref{exptelf_spk}:
\begin{multline*}
	\frac{(f_s^{p+1})^{(\ell)}(t)}{(f_s^{p+1})'(t)}-	\frac{(f_s^{p})^{(\ell)}(t)}{(f_s^{p})'(t)}\\
	=\sum_{i=2}^\ell\frac{\ell!}{i!}f_s^{(i)}(f_s^{p}(t))\sum_{(n_1,\ldots,n_i)\in S_{i,\ell}}\frac{1}{(f_s^{p+1})'(t)}\prod_{j=1}^i\frac{(f_s^{p})^{(n_j)}(t)}{n_j!}.
\end{multline*}
For all $1\le i\le \ell$, according to the induction hypothesis, there exists a positive function $C_{i,s}$ such that:  
\begin{align}\label{equivdurdur}
\sum_{(n_1,\ldots,n_i)\in S_{i,\ell}}\frac{1}{(f_s^{p+1})'(t)}\prod_{j=1}^i\frac{(f_s^{p})^{(n_j)}(t)}{n_j!}\sim C_{i,s}(t)r^{p(\ell-1)}e^{r^p(i-r)b(t)}.
\end{align}
Note that when $t$ goes to 0:
\[
f_s^{(i)}(t)=
\begin{cases}
t^{r-i}\frac{r!}{(r-i)!}\alpha_r(s)+o(t^{r-i}), & \mbox{if }i<r\\
i!\alpha_i(s)+o(1),&\mbox{ otherwise.}
\end{cases}
\]
Consequently, as $f_s^p(t)$ goes to 0 when $p$ goes to infinity, using \eqref{equivdurdur}:
\begin{align*}
& \frac{(f_s^{p+1})^{(\ell)}(t)}{(f_s^{p+1})'(t)}-	\frac{(f_s^{p})^{(\ell)}(t)}{(f_s^{p})'(t)}\\
&\sim r^{p(\ell-1)}\sum_{i=2}^{r-1}\frac{\ell!}{i!}\left(f_s^p(t)\right)^{r-i}\frac{r!}{(r-i)!}\alpha_r(s)C_{i,s}(t)e^{r^p(i-r)b(t)}\\
&+ r^{p(\ell-1)}\sum_{i=r}^{\ell}\ell!\alpha_i(s)C_{i,s}(t)e^{r^p(i-r)b(t)}\\
&\sim r^{p(\ell-1)}\alpha_r(s)\ell!\sum_{i=2}^{r}\binom{r}{i}K_{s,0}(t)^{r-i}C_{i,s}(t)
=:\breve{K}_s(t)r^{p(\ell-1)},
\end{align*}
where $\breve{K}_s$ is a positive function.\\
Since $r^{\ell-1}>1$, the series diverge and the partial sums are equivalent, which gives:
\begin{align*}
	\frac{\left(f_s^{p}\right)^{(\ell)}(t)}{\left(f_s^{p}\right)'(t)}&
	\underset{p\to+\infty}{\sim} \sum_{j=1}^{p-1}r^{j(\ell-1)}\breve{K}_s(t),
	\underset{p\to+\infty}{\sim} r^{p(\ell-1)}\breve{K}_s(t),
\end{align*}
and we conclude using \eqref{equivdur} for $\ell=1$. 
\end{proof}

Now we can state the 
result concerning the limit with the penalization function $H_\ell(Z_p)s^{M_p}t^{Z_p-\ell}$ for all $\ell\in \N$ when $\bp(0)=0$.
\begin{theo}\label{thm:cond_expo_r}
\begin{enumerate}
\item 	Let $\bp$ be an offspring distribution satisfying \eqref{condp}, $\bp(0)=0$ and which admits a moment of order $\ell\in\NN$ and let $\bq$ be mark function $\bq$ satisfying \eqref{condq}. Let $r=\min\{k\in\N,\  \bp(k)>0\}$. Let $t\in(0,1]$ and $s\in(0,1)$ be fixed.  We have for all $n\in\N$, and all $\Lambda_n\in \cF_n$,
	\begin{equation}
		\frac{\E\left[\indic_{\Lambda_n}H_\ell(Z_{n+p})s^{M_{n+p}}t^{Z_{n+p}}\right]}{\E\left[H_\ell(Z_{n+p})s^{M_{n+p}}t^{Z_{n+p}}\right]}\underset{p\to +\infty}{\longrightarrow}
			\E\left[\indic_{\Lambda_n}B_{s,n,0,r}\right],
				\end{equation}
	with 
	\[
	B_{s,n,0,r}=\begin{cases}
	s^{M_n}\alpha_1(s)^{-n}\indic_{Z_n=1} & \text{if }r=1,\\
	s^{M_n}\alpha_r(s)^{\frac{-\left(r^n-1\right)}{r-1}}\indic_{Z_n=r^{n}} & \text{if }r>1.
	\end{cases}
	\]
		\item The probability measure $\Q_{s,0,r}$ defined on $(\T^*,\cF_\infty)$ by
\[\forall n\in\mathbb N, \Lambda_n\in\cF_n,\quad\Q_{s,0,r}(\Lambda_n)=\E\left[\indic_{\Lambda_n}B_{s,n,0,r}\right],\]
is the distribution of a regular $r$-ary tree whose each node is marked, independently of each others, with probability
\[
\frac{s\bq(r)}{s\bq(r)+1-\bq(r)}\cdot
\]
\end{enumerate}
\end{theo}

 \begin{proof}
We only prove part (1). The proof of (2) follows the same lines as in the proof of Theorem \ref{pensMntZn1}, (2).

\medskip
We begin with the case $\ell=0$ and $r=1$.  Using \eqref{expespsMptZp}, we obtain:
\begin{equation}\label{TCD}
	0\le \frac{\E\left[s^{M_{n+p}}t^{Z_{n+p}}|\cF_n\right]}{\E\left[s^{M_{n+p}}t^{Z_{n+p}}\right]}=\frac{s^{M_n}f_s^{p}(t)^{Z_n}}{f_s^{n+p}(t)}\le \frac{s^{M_n}f_s^{p}(t)^{Z_n-1}}{(f_s^n)'(0)}\leq \frac{1}{(f_s^n)'(0)},
\end{equation}
since the mean value inequality, applied to $f_s^n$ on $[0,1]$, gives: 
\[f_s^{n+p}(t)=f_s^n(f_s^{p}(t))-f_s^n(0)\ge (f_s^n)^\prime(0)f_s^{p}(t).\] 
As $0<(f_s^n)'(0)=\E\left[s^{M_n}\indic_{\{Z_n=1\}}\right]<1$, we can use \eqref{TCD} to apply Lebesgue’s dominated convergence theorem and we just study the a.s. limit to obtain our results.

For $r=1$, according to Lemma 10 in \cite{athreya_branching_2004} and taking the notations in the proof of Lemma \ref{equivf_sp^k}, we can define:
\begin{equation*}
	Q(t)=\int_{\kappa(s)}^t Q'(x)\mathrm{d}x,~t\in[0,1],
\end{equation*}
such that $\lim_{p\to +\infty}Q_p(t)=Q(t)$ and $Q$ is a positive function.

Since $\kappa(s)=0$, and thanks to the expression \eqref{dfQp(t)} we have 
\begin{equation}\label{expequifps0}
	f^p_s(t) \underset{p\to+ \infty}{\sim}f_s'(0)^pQ(t),
\end{equation}
with 
$f_s'(0)=\E[s^{M_1}\indic_{Z_1=1}]$. 
Thereby:
\begin{align*}
\frac{\E\left[s^{M_{n+p}}t^{Z_{n+p}}|\cF_n\right]}{\E\left[s^{M_{n+p}}t^{Z_{n+p}}\right]}
		&\underset{p\to+\infty}{\sim} s^{M_n}f_s'(0)^{p(Z_n-1)-n}Q(t)^{Z_n-1}\\
		&\underset{p\rightarrow\infty}{\longrightarrow}s^{M_n}f_s'(0)^{-n}\indic_{Z_n=1}, 
\end{align*}
since $s<1$, $f_s'(0) 
<1$, giving us our result.

\medskip
Henceforth, we consider $\ell=0$ and $r>1$. According to \eqref{expespsMptZp} and Lemma \ref{expfspk0}, we get:
\begin{multline*}
	\frac{\E\left[s^{M_{n+p}}t^{Z_{n+p}}\left|\mathscr F_n\right.\right]}{\E\left[s^{M_{n+p}}t^{Z_{n+p}}\right]}=	\frac{s^{M_n}f_s^{p}(t)^{Z_n}}{f_s^{n+p}(t)}\\
	\underset{p\to+\infty}{\sim} s^{M_n}\alpha_r^{\frac{-1}{r-1}(Z_n-1)}e^{r^{n+p}b(t)(Z_nr^{-n}-1)}\underset{p\rightarrow\infty}{\longrightarrow}s^{M_n}\alpha_r^{\frac{-(r^n-1)}{r-1}}\indic_{Z_n=r^{n}}
\end{multline*} 
since $b(t)<0$.
 To conclude this case, Lemma \ref{expfspk0} ensures that for $p$ large enough:
 		\[\frac{1}{2}\le \frac{f_s^p(t)}{\alpha_r(s)^{\frac{-1}{r-1}}e^{r^pb(t)}}\le \frac{3}{2}\]
implying that for $p$ large enough, as $Z_n\ge r^n$:
\[\frac{\E\left[s^{M_{n+p}}t^{Z_{n+p}}\left|\mathscr F_n\right.\right]}{\E\left[s^{M_{n+p}}t^{Z_{n+p}}\right]}=\frac{s^{M_n}f_s^{p}(t)^{Z_n}}{f_s^{n+p}(t)}\le \frac{\left(f_s^{p}(t)\right)^{r^n}}{f_s^{n+p}(t)}\le 2\left(\frac{3}{2}\right)^{r^n}\alpha_r(s)^{\frac{-(r^n-1)}{r-1}},\]
thus we can apply dominated convergence. 

\medskip
Now, we suppose $\ell\geq1$ and $r=1$. According to Lemma \ref{equivf_sp^k}, \eqref{expespsMptZpderk} and  \eqref{expequifps0} we have:
\begin{align*}
	&\frac{\E\left[H_\ell(Z_{n+p})s^{M_{n+p}}t^{Z_{n+p}-\ell}|\cF_n\right]}{f_s^{\prime}(0)^{p}}\\
	&=\frac{s^{M_n}}{f_s^{\prime}(0)^{p}}\sum_{i=1}^\ell H_i(Z_n)f_s^{p}(t)^{Z_n-i}{\sum_{(n_1,\ldots,n_i)\in S_{i,\ell}}}\prod_{j=1}^i\frac{\left(f_s^{p}\right)^{(n_j)}(t)}{n_j!}\\
	&\sim \frac{s^{M_n}}{f_s^{\prime}(0)^{p}}\sum_{i=1}^\ell H_i(Z_n)\left(f_s'(0)^{p}Q(t)\right)^{Z_n-i}{\sum_{(n_1,\ldots,n_i)\in S_{i,\ell}}}\prod_{j=1}^i\frac{f_s'(0)^{p}C_{n_j}(t)}{n_j!}\\
	&= s^{M_n}f_s'(0)^{p(Z_n-1)}\sum_{i=1}^\ell H_i(Z_n)Q(t)^{Z_n-i}{\sum_{(n_1,\ldots,n_i)\in S_{i,\ell}}}\prod_{j=1}^i\frac{C_{n_j}(t)}{n_j!}\\
	&\underset{p\rightarrow\infty}{\longrightarrow}s^{M_n}\frac{C_\ell(t)}{\ell!}\mathds{1}_{Z_n=1}.
\end{align*}
The rest of the proof is the same as the one of Theorem \ref{pensMntZnderk} as the following upper bound remains true in our case:
\begin{align*}
0\le \frac{\E\left[H_\ell(Z_{n+p})s^{M_{n+p}}t^{Z_{n+p}-\ell}|\cF_n\right]}{f_s^{\prime}(0)^{p}}\le C\sum_{i=1}^\ell H_i(Z_n).
\end{align*}

Finally, we consider the case $\ell\geq1$ and $r>1$. According to Lemma \ref{expfspk0}, we have for every $1\leq i\leq \ell$ and every $(n_1,\ldots,n_i)\in S_{i,\ell}$, when $p\to+\infty$:
\begin{align*}
\sum_{(n_1,\ldots,n_i)\in S_{i,\ell}}\prod_{j=1}^i\frac{\left(f_s^{p}\right)^{(n_j)}(t)}{n_j!}\sim r^{p\ell} e^{r^pib(t)}\tilde{K_i}(t),
\end{align*}
with $\tilde{K_i}(t)$ a positive function, since all the terms in the sum are positive and of the same order. According to \eqref{expespsMptZpderk} and \eqref{expfspk00}, when $p$ goes to infinity:
\begin{align*}
	& \frac{\E\left[H_\ell(Z_{n+p})s^{M_{n+p}}t^{Z_{n+p}-\ell}|\cF_n\right]}{r^{p\ell}e^{r^{m+p}b(t)}}\\
	&\sim s^{M_n}\sum_{i=1}^\ell H_i(Z_n)\left(e^{r^{p}b(t)}K_0(t)\right)^{Z_n-i} e^{r^{p}(i-r^n)b(t)}\tilde{K_i}(t)\\
	&\sim s^{M_n}e^{r^{p}(Z_n-r^n)b(t)}\sum_{i=1}^\ell H_i(Z_n)K_0(t)^{Z_n-i}\\
	&\underset{p\rightarrow\infty}{\longrightarrow} \mathds{1}_{Z_n=r^n}s^{M_n}\sum_{i=1}^\ell H_i(r^n)K_0(t)^{r^n-i}=:s^{M_n}a_n\mathds{1}_{Z_n=r^n},
\end{align*}
where $a_n$ is a positive constant. Note that for $p$ large enough, as $Z_n\ge r^n$:
\begin{align*}
	& \frac{\E\left[H_\ell(Z_{n+p})s^{M_{n+p}}t^{Z_{n+p}-\ell}|\cF_n\right]}{r^{p\ell}e^{r^{n+p}b(t)}}\\
	& \le \frac{s^{M_n}}{r^{p\ell}e^{r^{n+p}b(t)}}\sum_{i=1}^\ell H_i(Z_n)f_s^{p}(t)^{r^n-i}{\sum_{(n_1,\ldots,n_i)\in S_{i,\ell}}}\prod_{j=1}^i\frac{\left(f_s^{p}\right)^{(n_j)}(t)}{n_j!} \\
	 &\le s^{M_n}\left(\frac{3}{2}\right)^{r^n}\sum_{i=1}^\ell H_i(Z_n)K_0(t)^{r^n-i} {\sum_{(n_1,\ldots,n_i)\in S_{i,\ell}}}\prod_{j=1}^i\frac{K_{n_j}}{n_j!}\in L^1
\end{align*}
as $Z_n\in L^\ell$. As a result, Lebesgue’s dominated convergence theorem gives:
\begin{align*}
	\frac{\E\left[\mathds{1}_{\Lambda_n}H_\ell(Z_{n+p})s^{M_{n+p}}t^{Z_{n+p}-\ell}\right]}{r^{p\ell}e^{r^{n+p}b(t)}}=\E\left[\mathds{1}_{\Lambda_n}s^{M_n}a_n\mathds{1}_{Z_n=r^n}\right].
	\end{align*}
	The previous formula is true for $n=0$ and $\Lambda_0=\T^*$, consequently:
\begin{align*}
\lim_{p\rightarrow+\infty}\frac{\E\left[\mathds{1}_{\Lambda_n}H_\ell(Z_{n+p})s^{M_{n+p}}t^{Z_{n+p}-\ell}\right]}{\E\left[H_\ell(Z_{n+p})s^{M_{n+p}}t^{Z_{n+p}-\ell}\right]}& =\E\left[\mathds{1}_{\Lambda_n}s^{M_n}r^{-n\ell}\frac{a_n}{a_0}\mathds{1}_{Z_n=r^n}\right]\\
& =:\E\left[\mathds{1}_{\Lambda_n}s^{M_n}r^{-n\ell}b_n\mathds{1}_{Z_n=r^n}\right].
\end{align*}
Since $s^{M_n}r^{-n\ell}b_n\indic_{Z_n=r^{n}}$ is a martingale with mean 1, $b_n=r^{n\ell}\alpha_r(s)^{\frac{-(r^n-1)}{r-1}}$.
 \end{proof}
 
 \subsection{Case $s=0$}
 
 For the sake of completeness, we consider the case $s=0$, that is we penalize by the quantity
 \[
 H_\ell(Z_{p})\indic_{M_p=0}t^{Z_p}.
 \]
As the proofs are very close to that of the previous sections,
we only state the main results and give some ideas for the proofs.

Here, we introduce $\psi(t):=\E\left[\indic_{M_1=0}t^{Z_1}\right]$, $\tilde r:=\min\{j>0;~\bp(j)(1-\bq(j))>0\}$ and $\tilde{\kappa}$ the smallest positive fixed point of $\psi$.

\medskip
 We begin with the case $\tilde r=0$, where $0<\tilde{\kappa}<1$ and we obtain:

\begin{theo}\label{pen0MntZn}
	Let $\bp$ be an offspring distribution satisfying \eqref{condp} that admits a moment of order $\ell\in\NN$ and let $\bq$ be a mark function satisfying \eqref{condq} such that $\tilde r=0$, and let $t\in(0,1]$. We have for all $n\in\N$, $\Lambda_n\in\cF_n$:
	\begin{equation}
		\frac{\E\left[\indic_{\Lambda_n}H_\ell(Z_{n+p})\indic_{M_{n+p}=0}t^{Z_{n+p}}\right]}{\E\left[H_\ell(Z_{n+p})\indic_{M_{n+p}=0}t^{Z_{n+p}}\right]}\underset{p\to +\infty}{\longrightarrow}
		\left\{
		\begin{array}{ll}
			\E\left[\indic_{\Lambda_n}B_{0,n,1}\right],& \text{ if }\ell=0\\ 
			\\
			\E\left[\indic_{\Lambda_n}B_{0,n,2}\right],& \text{ otherwise },
		\end{array}
		\right. 
	\end{equation}
	where $B_{0,n,1}:=\tilde{\kappa}^{Z_n-1}\indic_{M_n=0}$ and  $B_{0,n,2}:=\frac{Z_n\tilde{\kappa}^{Z_n-1}}{\psi'(\tilde{\kappa})^{n}}\indic_{M_n=0}$.
	\end{theo}
The proof of this result is based on that of Theorem \ref{pensMntZn1} and Theorem \ref{pensMntZnderk}, and the following asymptotic behavior: 
	\begin{equation}\label{expequivpsip^k}
	\forall t\in[0,1],\, (\psi^p)^{(\ell)}(t)\underset{p\to+\infty}{\sim} \psi'(\tilde{\kappa})^pC_\ell(t).
	\end{equation} 
	where $C_\ell$ is a positive function (the proof is similar to that of Lemma \ref{equivf_sp^k}).\\
Under the probability $\mathbb Q_{0,i}$ defined on $(\T^*, \mathscr F)$ by $\Q_{0,i}(\Lambda_n)=\E\left[\indic_{\Lambda_n}B_{0,n,i}\right]$, 
for:
\begin{itemize}
\item $i=1$, $\tau^*$ is a MGW with the reproduction-marking law:
\[p_{0,1}(k,\eta)=\bp_0(k,\eta)\tilde{\kappa}^{k-1}\delta_{\eta 0};\]
\item $i=2$, $\tau^*$ is a two-typed random marked tree, there are normal and special nodes. The reproduction-marking law of the normal ones is $p_{0,1}$ and the special ones' is:
\[p_{0,2}(k,\eta)=\bp_0(k,\eta)\frac{k\tilde{\kappa}^{k-1}}{\psi'(\tilde{\kappa})}\delta_{\eta 0},\] 
and at each generation, there is a unique special node chosen uniformly among all the individuals of this generation.
\end{itemize}

\medskip
Henceforth, we consider the case $\tilde r\geq1$. We begin to present a lemma similar to Lemma \ref{expfspk0} for the case $\tilde r>1$.
\begin{lem}\label{exppsipk0}
	Suppose $ \tilde r\geq2$. For every $\ell\in\N$, and every $t\in(0,1]$, there exists a positive constant $K_\ell(t)$, such that
	\begin{equation*}
		(\psi^p)^{(\ell)}(t)=K_\ell(t)r^{p\ell}e^{\tilde r^pb(t)}(1+o(1)),
	\end{equation*}
	where
	\begin{equation*}
		b(t):= \ln(t)+\sum_{j=0}^{+\infty}\tilde r^{-(j+1)}\ln\left(\frac{\psi^{j+1}(t)}{\psi^j(t)^{\tilde r}}\right).
	\end{equation*}
\end{lem}
Thereby we obtain the following theorem:
\begin{theo}\label{pen0MntZnp0}
Let  $\bp$ be an offspring distribution satisfying \eqref{condp} that admits a moment of order $\ell\in\NN$,  and let $\bq$ be a mark function satisfying \eqref{condq} such that $\tilde r\ge 1$, and let $t\in(0,1]$. We have for all $n\in\N$, $\Lambda_n\in\cF_n$:
	\begin{multline*}
		\frac{\E\left[\indic_{\Lambda_n}H_\ell(Z_{n+p})\indic_{M_{n+p}=0}t^{Z_{n+p}}\right]}{\E\left[H_\ell(Z_{n+p})
		\indic_{M_{n+p}=0}t^{Z_{n+p}}\right]}\\
		\underset{p\to +\infty}{\longrightarrow}
		\begin{cases}
			\E\left[\indic_{\Lambda_n}\bp_0(1,0)^{-n}\indic_{M_n=0,Z_n=1}\right],& \text{if }\tilde r=1,\\ 
			\E\left[\indic_{\Lambda_n}\bp_0(r,0)^{\frac{-\left(\tilde r^n-1\right)}{\tilde r-1}}\indic_{M_n=0,Z_n=\tilde r^{n}}\right],& \text{if }\tilde r>1.
		\end{cases}
	\end{multline*}
\end{theo}

The probabilities obtained with these martingales are Dirac masses at $\tilde r$-ary regular trees with no marks.

	\bibliographystyle{abbrv}
	\bibliography{biblio21}
	
\end{document}